\documentclass[a4paper,11pt,reqno,twoside]{article}
\usepackage{amssymb}
\usepackage[frenchb,english]{babel}
\usepackage{amscd}
\usepackage{amsmath,amsthm,amsfonts,amssymb,graphicx}
\usepackage[colorlinks,linkcolor=blue,anchorcolor=blue,citecolor=green]{hyperref}
\usepackage{mathrsfs}
\usepackage{fancyhdr}
\usepackage{subfigure}
\usepackage{bm}
\usepackage{multicol}
\usepackage{picins}
\usepackage{abstract}

\makeatother

\begin{document}
\theoremstyle{plain}
\newtheorem{Definition}{Definition}[section]
\newtheorem{Proposition}{Proposition}[section]
\newtheorem{Property}{Property}[section]
\newtheorem{Theorem}{Theorem}[section]
\newtheorem{Lemma}[Theorem]{\hspace{0em}\bf{Lemma}}
\newtheorem{Corollary}[Theorem]{Corollary}
\newtheorem{Remark}{Remark}[section]

\setlength{\oddsidemargin}{ 1cm}  
\setlength{\evensidemargin}{\oddsidemargin}
\setlength{\textwidth}{13.50cm}
\vspace{-.8cm}

\noindent  {\LARGE
\begin{center}
The first two coefficients of the Bergman function expansions\\ for Cartan-Hartogs domains
\end{center}
}

\vskip 20pt
\noindent\text{Zhiming Feng  }\\
\noindent\small {School of Mathematical and Information Sciences, Leshan Normal University, Leshan, Sichuan 614000, P.R. China } \\
\noindent\text{Email: fengzm2008@163.com}

\vskip 20pt

\normalsize \noindent\textbf{Abstract}\quad { Let $\phi$ be a globally defined real K\"{a}hler potential on a domain $\Omega\subset \mathbb{C}^d$, and $g_{F}$ be a K\"{a}hler  metric  on the Hartogs domain $ M=\{(z,w)\in \Omega\times\mathbb{C}^{d_0}: \|w\|^2<e^{-\phi(z)}\}$ associated with the K\"{a}hler
potential $\Phi_{F}(z,w)=\phi(z)+F(\phi(z)+\ln\|w\|^2)$.  Firstly, we obtain explicit formulas of the coefficients $\mathbf{a}_j\;(j=1,2)$ of the Bergman function expansion for the Hartogs domain $( M,g_F)$ in a momentum profile $\varphi$. Secondly, using explicit expressions of $\mathbf{a}_j\;(j=1,2)$, we obtain necessary and sufficient conditions for the coefficients $\mathbf{a}_j\;(j=1,2)$ to be constants. Finally, we obtain all the invariant complete K\"{a}hler metrics on Cartan-Hartogs domains such that their the coefficients $\mathbf{a}_j\; (j=1,2)$ of the Bergman function expansions are constants.
\smallskip\\\
\textbf{Key words:} K\"{a}hler metrics \textperiodcentered \; Coefficients of the Bergman function expansion \;\textperiodcentered \; Hartogs domains \;\textperiodcentered \; Cartan-Hartogs domains
\smallskip\\\
\textbf{Mathematics Subject Classification (2010):} 32Q15 \textperiodcentered \,  53C55

\setlength{\oddsidemargin}{-.5cm}  
\setlength{\evensidemargin}{\oddsidemargin}
\pagenumbering{arabic}
\renewcommand{\theequation}
{\arabic{section}.\arabic{equation}}

 \setcounter{equation}{0}
\section{{Introduction}}

Let $M$ be a domain in $\mathbb{C}^n$, $\phi$ be a K\"{a}hler potential on $M$, and $g$ be a
K\"{a}hler  metric  on $M$ associated with the K\"{a}hler form $\omega=\frac{\sqrt{-1}}{2\pi}\partial\overline{\partial}\phi$. Set
$$\mathcal{H}_{\alpha}=\left\{ f\in \textmd{Hol}(M): \int_{M}|f|^2e^{-\alpha \phi}\frac{\omega^n}{n!}<+\infty\right\},$$
where $\textmd{Hol}(M)$ denotes the space of holomorphic functions on $M$.  For $\alpha>0$,  let $K_{\alpha}(z,\overline{z})$ be the Bergman kernel (namely, the
reproducing kernel) of the Hilbert space $\mathcal{H}_{\alpha}$ if $\mathcal{H}_{\alpha}\neq \{0\}$, the Bergman function  on $M$ associated with the metric $g$  defined by
\begin{equation}\label{eq1.8}
 \varepsilon(\alpha;z)=e^{-\alpha \phi(z)}K_{\alpha}(z,\overline{z}),\;\; z\in M.
\end{equation}
The metric $g$ is called balanced when $\varepsilon(1;z)$ is constant. Balanced metric plays an important  role in the quantization  of a
K\"{a}hler manifold, see Berezin \cite{Berezin}, Cahen-Gutt-Rawnsley \cite{CGR}, Engli\v{s} \cite{E0}, Ma-Marinescu \cite{MM07}, Ma \cite{MA2010} and Lui\'c \cite{Lui}.

Under given some conditions for $(M,\omega)$, $\epsilon(\alpha;z)$ admits an asymptotic expansion as $\alpha\rightarrow +\infty$
\begin{equation}\label{eq1.6}
  \epsilon(\alpha;z)\sim\sum_{j=0}^{\infty} a_j(z)\alpha^{n-j},
\end{equation}
where the expansion coefficients $a_0, a_1, a_2$ in Lu \cite{Lu} and  Engli\v{s} \cite{E2} are given by
\begin{equation}\label{eq1.7}
\left\{  \begin{array}{ll}
    a_0 & =1, \\
    a_1 & = \frac{1}{2}k_g, \\
    a_2 & =\frac{1}{3}\triangle k_g+\frac{1}{24}|R_g|^2-\frac{1}{6}|\text{Ric}_g|^2+\frac{1}{8}k_g^2.
  \end{array}\right.
\end{equation}
 Here $k_g$, $\triangle_g$, $R_g$ and $\text{Ric}_g$  denote  the scalar curvature, the Laplace, the curvature tensor and the Ricci curvature associated with the metric $g$, respectively. If for all sufficiently large positive numbers $\alpha$, $\epsilon(\alpha;z)$ are constants in $z$ on $M$, Loi in \cite{Loi1}  has proved that there exists an asymptotic expansion on $(M,\omega)$ for $\varepsilon(\alpha;z)$ as \eqref{eq1.6} and \eqref{eq1.7}, and all coefficients $a_j$ are constants.  For graph theoretic formulas of coefficients $a_j$, see Xu \cite{X}. For the general reference of the Bergman function expansions, refer to Berezin \cite{Berezin}, Catlin \cite{Cat}, Zelditch \cite{Zeld}, Engli\v{s} \cite{E1}, Dai-Liu-Ma \cite{Dai-Liu-Ma}, Ma-Marinescu \cite{MM07,MM08,MM12},  Berman-Berndtsson-Sj\"ostrand \cite{B-B-S}, Hsiao \cite{HCY} and Hsiao-Marinescu \cite{HM2014}.

In \cite{Donaldson} Donaldson used the first coefficient $a_1$ in the expansion of the Bergman function to give conditions for the existence and uniqueness of constant scalar curvature K\"{a}hler metrics (cscK metrics). This work inspired many papers on the subject since then.

The main purpose of this paper is to study cscK metrics on Hartogs domains such that the coefficients $a_2$  of the Bergman function expansion are constants.  For the study of K\"ahler metrics with the constant coefficients $a_2$, see Loi \cite{Loi1}, Loi-Zuddas \cite{Loi-Zud}, Zedda \cite{Zed}, Feng-Tu \cite{FT},  Loi-Zedda \cite{Loi-Zed2} and Feng \cite{F1}.

In \cite{B-F-T} and \cite{FT2}, for Fock-Bargmann-Hartogs domains and Cartan-Hartogs domains,  we studied balanced metrics associated with K\"ahler forms
\begin{equation*}
  \omega=\frac{\sqrt{-1}}{2\pi}\partial\overline{\partial}\left(\nu\phi(z)-\ln(e^{-\phi(z)}-\|w\|^2)\right),\nu\geq 0.
\end{equation*}
From Lemma \ref{Le:7.1} below, we know that the K\"{a}hler forms of $\mathcal{G}$-invariant K\"{a}hler  metrics  on Cartan-Hartogs domains $\Omega(\mu,d_0)$ can be written as
$$\omega=\frac{\nu\sqrt{-1}}{2\pi}\partial\overline{\partial}\left(\phi(z)+F(e^{\phi(z)}\|w\|^2)\right),\;\nu>0.$$
So in this paper, we will study  K\"ahler metrics associated with K\"ahler forms
\begin{equation*}
 \omega=\frac{\sqrt{-1}}{2\pi}\partial\overline{\partial}\left(\phi(z)+F(\phi(z)+\ln\|w\|^2)\right)
\end{equation*}
on Hartogs domains
$$ M=\left\{(z,w)\in \Omega\times\mathbb{C}^{d_0}: \|w\|^2<e^{-\phi(z)}\right\}.$$

Now we give  main results of this paper, namely the following theorems.

\begin{Theorem}\label{apth:3.3}{
Let $g_{\phi}$ be a complete K\"{a}hler  metric  on the domain $\Omega$ associated with the K\"{a}hler form $\omega_{\phi}=\frac{\sqrt{-1}}{2\pi}\partial\overline{\partial}\phi$, where $\phi$ is a globally defined  K\"{a}hler potential on a domain $\Omega\subset \mathbb{C}^d$, defined a Hartogs domain
$$ M=\left\{(z,w)\in \Omega\times\mathbb{C}^{d_0}: \|w\|^2<e^{-\phi(z)}\right\}.$$
Set $g_F$ is a complete K\"{a}hler  metric  on the domain $M$ associated with the K\"{a}hler form $\omega_F=\frac{\sqrt{-1}}{2\pi}\partial\overline{\partial}\Phi_F$, here
$$\Phi_F(z,w)=\phi(z)+F(\phi(z)+\ln\|w\|^2),\;F(-\infty)=0.$$

Let $k_{g_{\phi}}$, $\Delta_{g_{\phi}}$, $\mathrm{Ric}_{g_{\phi}}$ and $R_{g_{\phi}}$ be the scalar curvature, the Laplace, the Ricci curvature and the curvature tensor on the domain $\Omega$ with respect to the metric $g_{\phi}$, respectively. Put
\begin{equation*}
 a_1=\frac{1}{2}k_{g_{\phi}},\;a_2=\frac{1}{3}\triangle_{g_{\phi}} k_{g_{\phi}}+\frac{1}{24}|R_{g_{\phi}}|^2-\frac{1}{6}|\mathrm{Ric}_{g_{\phi}}|^2+\frac{1}{8}k_{g_{\phi}}^2
\end{equation*}
and
\begin{equation*}
  \mathbf{a}_1=\frac{1}{2}k_{g_F},\;\mathbf{a}_2=\frac{1}{3}\triangle_{g_F} k_{g_F}+\frac{1}{24}|R_{g_F}|^2-\frac{1}{6}|\mathrm{Ric}_{g_F}|^2+\frac{1}{8}k_{g_F}^2.
\end{equation*}

 Then both $\mathbf{a}_1$ and $\mathbf{a}_2$ are constants on $M$ if and only if

 $(\mathrm{i})$
\begin{equation*}
 F(t)=-\frac{1}{A}\ln(1-e^t)
\end{equation*}
and $a_2=0$ for $d=1$, where $A=\frac{d_0-a_1}{d_0+1}$ is a positive constant.

 $(\mathrm{ii})$
\begin{equation*}
 F(t)=-\ln\left(1-e^t\right),
\end{equation*}
$a_1=-\frac{d(d+1)}{2}$ and $a_2=\frac{(d-1)d(d+1)(3d+2)}{24}$ for $d>1$.
}\end{Theorem}

Let $\Omega$ be a bounded symmetric domain in $\mathbb{C}^d$, that is, $\Omega$ is bounded and  there exists a holomorphic automorphism $\gamma_{z_0}$ for every $z_0\in \Omega$ such that $\gamma_{z_0}$ is involutive (i.e. $\gamma_{z_0}^2 = id$) and $z_0$ is an isolated fixed point of $\gamma_{z_0}$. If $\Omega$ can not be expressed as the product of two bounded symmetric domains, then $\Omega$ is called an irreducible bounded symmetric domain. If we put the Bergman metric on $\Omega$, then $\Omega$ is a Hermitian symmetric space. By selecting  proper coordinate systems, irreducible bounded symmetric domains can be realized as  circular convex domains. The following assume that irreducible bounded symmetric domains are circle convex domains.

Every Hermitian symmetric space of noncompact type can be realized as a bounded symmetric domain in some $\mathbb{C}^d$ by the Harish-Chandra embedding theorem.  In 1935, E. Cartan has showed that there exist only six types of irreducible bounded symmetric domains. They are four types of classical bounded
symmetric domains ($\Omega_I(m,n), \Omega_{II}(n), \Omega_{III}(n), \Omega_{IV}(n)$) and two exceptional domains ($\Omega_{V}(16), \Omega_{V}(27)$). So irreducible bounded symmetric domains are called Cartan domains.

For an irreducible bounded symmetric domain $\Omega$, we denote by $r, a, b, d, p$ and $N$,  the rank, the characteristic multiplicities, the dimension, the genus, and the generic norm of $\Omega$, respectively; thus
\begin{equation*}
  p=(r-1)a+b+2,\;d=\frac{r(r-1)}{2}a+br+r.
\end{equation*}
For convenience, we list  the characteristic multiplicities, the rank, and the generic norm $N$ for the classical domain $\Omega$ according to its type as the following table.
\begin{equation*}
  \begin{tabular}{|c|c|c|}
     \hline $\Omega$ & $(a,b,r)$ & $N$  \\
     \hline $\Omega_I(m,n)=\{z\in \mathcal{M}_{m,n}:I-z{\overline z}^t>0 \} $ ($1\leq m\leq n)$ & $(2,n-m,m)$ & $\det(I-z{\overline z}^t)$   \\
     \hline $\Omega_{II}(2n)=\{z\in \mathcal{M}_{2n,2n}: z^t=-z,I-z{\overline z}^t>0 \}$ ($n\geq 3$) & $(4,0,n)$ & $\sqrt{\det(I-z{\overline z}^t)}$  \\
     \hline $\Omega_{II}(2n+1)=\{z\in \mathcal{M}_{2n+1,2n+1}: z^t=-z,I-z{\overline z}^t>0 \}$ ($n\geq 2$) & $(4,2,n)$ & $\sqrt{\det(I-z{\overline z}^t)}$  \\
     \hline $\Omega_{III}(n)=\{z\in \mathcal{M}_{n,n}:z^t=z, I-z{\overline z}^t>0 \}$ ($n\geq 2$) & $(1,0,n)$ & $\det(I-z{\overline z}^t)$   \\
     \hline $\Omega_{IV}(n)=\{z\in \mathbb{C}^{n}:1-2z{\overline z}^t+|zz^t|^2>0,  z{\overline z}^t<1\} $ ($n\geq 5$) & $(n-2,0,2)$  & $1-2z{\overline z}^t +|zz^t|^2$   \\
     \hline
   \end{tabular}
\end{equation*}
The above, $\mathcal{M}_{m,n}$ denotes the set of all $m\times n$ matrices $z=(z_{ij})$ with complex entries,  ${\overline z}$ is the
complex conjugate of the matrix $z$, ${z}^t$ is the transpose of the matrix $z$, $I$ denotes the identity matrix, and $z>0$ indicates that the square matrix $z$ is positive definite. For the reference of the irreducible bounded symmetric domains, refer to Hua \cite{Hua} and Faraut-Kaneyuki-Kor\'{a}nyi-Lu-Roos \cite{FKKLR}.

For the Cartan domain $\Omega$ in $\mathbb{C}^d$, a positive real number $\mu$ and a positive integer number $d_0$, let
\begin{equation}\label{e7.2}
 \phi(z):=-\mu\ln N(z,\overline{z}),
\end{equation}
the Cartan-Hartogs domain $\Omega(\mu,d_0)$ is defined by
\begin{equation*}\label{e7.1}
  \Omega(\mu,d_0):=\left\{(z,w)\in  \Omega\times \mathbb{C}^{d_0}\subset \mathbb{C}^{d}\times \mathbb{C}^{d_0}
  : \|w\|^2<e^{-\phi(z)} \right\},
\end{equation*}
where $N$ is the generic norm of $\Omega$, and $\|\cdot\|$ the standard Hermitian norm in $\mathbb{C}^{d_0}$. Below we assume $\frac{\partial^2\phi}{\partial z^t\partial\bar{z}}(0)=\mu I_d$.

From Lemma 3.1 of Ahn-Byun-Park \cite{ABP}, the Cartan-Hartogs domain $\Omega(\mu,d_0)$ is homogeneous if and only if $\Omega$ is the unit ball in $\mathbb{C}^{d}$ and $\mu=1$. For the general reference of Cartan-Hartogs domains, see  Ahn-Byun-Park \cite{ABP},  Feng-Tu \cite{FT},  Loi-Zedda \cite{LZ}, Wang-Yin-Zhang-Roos  \cite{RWYZ}, Zedda \cite{Zed} and references therein.

Let $\mbox{\rm Aut}(\Omega(\mu,d_0))$ be the holomorphic automorphism group of $\Omega(\mu,d_0)$. Assume that $\Omega(\mu,d_0)$ is not the unit ball, from  Wang-Yin-Zhang-Roos \cite{RWYZ}, Ahn-Byun-Park \cite{ABP} and Tu-Wang \cite{TW}, $\mbox{\rm Aut}(\Omega(\mu,d_0))$  be exactly the set of all mappings
$\Upsilon$:
\begin{equation}\label{e7.3}
\Upsilon(z,w)=\left(\gamma(z),\psi(z)wU \right),\;U\in \mathcal{U}(d_0),\;(z,w)\in \Omega(\mu,d_0),
\end{equation}
where $\gamma\in \mbox{\rm Aut}(\Omega)$, $z_0 =\gamma^{-1}(0)$,
\begin{equation}\label{e7.4}
  \psi(z)=\frac {N(z_0,\bar z_0)^{{\mu}/{2}}} {N(z,\bar z_0)^{{\mu}}}
\end{equation}
and $\mathcal{U}(d_0)$ is the unitary group of degree $d_0$ which consisting of all $d_0\times d_0$ unitary matrices.

As corollaries of Theorem \ref{apth:3.3}, we obtain all the invariant complete  metrics on Cartan-Hartogs domains such that their
the coefficients $\mathbf{a}_j\; (j=1,2)$ of the Bergman function expansions are constants.

\begin{Theorem}\label{Th:1.1}{
For a given positive integer $d_0$ and a positive real number $\mu$, let
 $$ \Omega(\mu,d_0):=\left\{(z,w)\in \Omega\times\mathbb{\mathbb{B}}^{d_0}: \|w\|^2<N(z,\overline{z})^{\mu}\right\}$$
 be the Cartan-Hartogs domain, where $N(z,\overline{z})$ is the generic norm of an irreducible bounded symmetric domain $\Omega$ in $\mathbb{C}^{d}$.

Let $\mathcal{G}$ be the group of holomorphic automorphism mappings generated by \eqref{e7.3} on the Cartan-Hartogs domain $\Omega(\mu,d_0)$,
 $g$ be a $\mathcal{G}$-invariant complete K\"{a}hler  metric  on the domain $\Omega(\mu,d_0)$ associated with  the K\"{a}hler form
$\omega_g=\frac{\sqrt{-1}}{2\pi}\partial\overline{\partial}\Phi$, that is
\begin{equation*}
  \partial\bar{\partial}(\Phi\circ\Upsilon)=\partial\bar{\partial}\Phi,\;\forall \;\Upsilon\in \mathcal{G}.
\end{equation*}

Then the first two coefficients $\mathbf{a}_j\; (j=1,2)$ of the Bergman function expansion for $(\Omega(\mu,d_0),\omega_g)$ are constants  if and only if

$\mathrm{(i)}$ $\Omega=\mathbb{B}^d:=\{z\in\mathbb{C}^d:\|z\|^2<1\}$,  and
\begin{equation*}
 \omega_g=-\frac{\nu\sqrt{-1}}{2\pi}\partial\overline{\partial}\ln(1-\|z\|^2-\|w\|^2),\;\nu>0
\end{equation*}
for $d>1$.

$\mathrm{(ii)}$ $\Omega=\mathbb{B}:=\{z\in\mathbb{C}:|z|<1\}$ and
\begin{equation*}
 \omega_g=-\frac{\nu\sqrt{-1}}{2\pi}\partial\overline{\partial}\left\{\mu\ln(1-|z|^2)+\frac{(d_0+1)\mu}{d_0\mu+1}\ln\left(1-\frac{\|w\|^2}{(1-|z|^2)^{\mu}}\right)\right\},\;\nu>0
\end{equation*}
for $d=1$.

 }\end{Theorem}

\begin{Remark}

$\mathrm{(i)}$ From \cite{LZ}, for $\nu>d$, K\"{a}hler metrics associated with
$$\omega=-\frac{\nu\sqrt{-1}}{2\pi}\partial\overline{\partial}\ln(1-\|z\|^2)$$
 are balanced on unite balls $\mathbb{B}^d$.

$\mathrm{(ii)}$ Using methods of \cite{B-F-T}, we can prove that for $\nu>\frac{1}{\mu}+d_0$, K\"{a}hler metrics associated with
$$ \omega=-\frac{\nu\sqrt{-1}}{2\pi}\partial\overline{\partial}\left\{\mu\ln(1-|z|^2)+\frac{(d_0+1)\mu}{d_0\mu+1}\ln\left(1-\frac{\|w\|^2}{(1-|z|^2)^{\mu}}\right)\right\}$$
are balanced on Cartan-Hartogs domains
$$\left\{(z,w)\in \mathbb{B}\times\mathbb{C}^{d_0}:\|w\|^2<(1-|z|^2)^{\mu}\right\}.$$
\end{Remark}

To prove Theorem \ref{apth:3.3}, let
\begin{equation*}
  x=F'(t),\;\varphi(x)=F''(t),
\end{equation*}
then the scalar curvature of $g_F$ is given by a linear second-order differential expression in $\varphi(x)$. Consequently, the $\varphi(x)$ is an explicit quadratic  function of $x$ when both $\mathbf{a}_1$ and $\mathbf{a}_2$ are constants. This method referred to as the momentum construction (see \cite{Calabi}, \cite{Hwang-Singer}), the function $\varphi(x)$ is called the momentum profile of $\omega_F$.

The paper is organized as follows. In Section 2, by calculating the scalar curvature, the squared norm of the Ricci curvature tensor, the squared norm of the curvature tensor, and the Laplace of the scalar curvature, we obtain explicit expressions of the coefficients $\mathbf{a}_j\;(j=1,2)$ of the Bergman function expansion for $(M,g_F)$ in the momentum profile $\varphi(x)$. In Section 3, using the expressions of the coefficients $\mathbf{a}_j\;(j=1,2)$ of the Bergman function expansion for $(M,g_F)$, we obtain an explicit expression of the function $F$ when $\mathbf{a}_1$ and $\mathbf{a}_2$ on Hartogs domain $(M,g_F)$  are constants, thus we obtain necessary and sufficient conditions for the coefficients $\mathbf{a}_j\;(j=1,2)$ to be constants. In Section 4, we first give the general expressions of holomorphic invariant K\"{a}hler metrics on Cartan-Hartogs domain $\Omega(\mu,d_0)$, and then give all invariant complete K\"{a}hler metrics such that their the coefficients $\mathbf{a}_j\; (j=1,2)$ of the Bergman function expansions are constants.

\setcounter{equation}{0}
\section{The first two coefficients of the Bergman function expansions for Hartogs domains}

The following we first compute the scalar curvature, the squared norm of the Ricci curvature tensor,  and the Laplace of the scalar curvature for the metric $g_F$ on $M$. Secondly, we obtain an expression of the squared norm of the curvature tensor. Finally we get expressions of the coefficients $\mathbf{a}_j\;(j=1,2)$ of the Bergman function expansion on $(M,g_F)$. As applications, we obtain the necessary and sufficient conditions for the coefficients  $\mathbf{a}_j\;(j=1,2)$ to be constants when $\varphi(x)=x(1+x)$ or $\varphi(x)=Ax^2+x$ with $d=1$ .

To prove the Theorem \ref{Th:5.1}, we need the following Lemma \ref{Le:4.1} and Lemma \ref{Le:4.2}.

\begin{Lemma}\label{Le:4.1}{
Let $\phi$ be a globally defined real K\"{a}hler potential on a domain $\Omega$, $F$ be a real function on $[-\infty,0)$ and
$$\Phi_F(z,w)=\phi(z)+F(t),$$
where $z\in \mathbb{C}^d,w\in\mathbb{C}^{d_0}$, and
\begin{equation*}\label{e4.1}
  t:=\phi(z)+\ln r^2,\;r^2:=\|w\|^2.
\end{equation*}

Set
\begin{equation*}\label{e4.2}
Z=(z,w),\;  T\equiv(T_{i\bar{j}}):=\frac{\partial^2\Phi_F}{\partial Z^t\partial \overline{Z}}\equiv\left(
                                                                \begin{array}{cc}
                                                                  T_{1} & T_{2} \\
                                                                  T_{3} & T_{4} \\
                                                                \end{array}
                                                              \right)\equiv\left(
                                                                     \begin{array}{cc}
                                                                      \frac{\partial^2\Phi_F}{\partial z^t\partial \bar z}  & \frac{\partial^2\Phi_F}{\partial z^t\partial \bar w} \\
                                                                       \frac{\partial^2\Phi_F}{\partial w^t\partial \bar z} & \frac{\partial^2\Phi_F}{\partial w^t\partial \bar w} \\
                                                                     \end{array}
                                                                   \right),
\end{equation*}
and
\begin{equation*}\label{e4.3}
T^{-1}\equiv(T^{\bar{i}j}):=\left(
         \begin{array}{cc}
           (T^{-1})_{1} & (T^{-1})_{2} \\
           (T^{-1})_{3} & (T^{-1})_{4} \\
         \end{array}
       \right).
\end{equation*}
Then
\begin{equation}\label{e4.4}
  T_1=(1+F')\frac{\partial^2\phi}{\partial z^t\partial \bar z}+F'' \frac{\partial\phi}{\partial z^t}\frac{\partial\phi}{\partial \bar{z}},\;\; T_2=\frac{F''}{r^2}\frac{\partial\phi}{\partial z^t}w,
\end{equation}
\begin{equation}\label{e4.6}
  T_3=\frac{F''}{r^2}\bar{w}^t\frac{\partial\phi}{\partial \bar{z}},\;\;  T_4=\frac{F'}{r^2}I_{d_0}+\frac{F''-F'}{r^4}\bar{w}^tw,
\end{equation}
\begin{equation}\label{e4.8}
   \det T=\frac{1}{r^{2d_0}}(F')^{d_0-1}F''(1+ F')^d\det(\frac{\partial^2\phi}{\partial z^t\partial \bar z}),
\end{equation}
\begin{equation}\label{e4.9}
  (T^{-1})_1=\frac{1}{1+F'}\left(\frac{\partial^2\phi}{\partial z^t\partial \bar z}\right)^{-1},
\end{equation}
\begin{equation}\label{e4.10}
  (T^{-1})_2=-\frac{1}{1+F'}\left(\frac{\partial^2\phi}{\partial z^t\partial \bar z}\right)^{-1}\frac{\partial\phi}{\partial z^t}w,
\end{equation}
\begin{equation}\label{e4.11}
  (T^{-1})_3=-\frac{1}{1+F'}\bar{w}^t\frac{\partial\phi}{\partial \bar{z}}\left(\frac{\partial^2\phi}{\partial z^t\partial \bar z}\right)^{-1},
\end{equation}
\begin{eqnarray}
\label{e4.12}    (T^{-1})_4   & =& \frac{r^2}{F'}I_{d_0}+\left(\frac{1}{F''}-\frac{1}{F'}\right)\overline{w}^tw+\frac{1}{1+F'}\bar{w}^t\frac{\partial\phi}{\partial \bar{z}}\left(\frac{\partial^2\phi}{\partial z^t\partial \bar z}\right)^{-1}\frac{\partial\phi}{\partial z^t}w.
\end{eqnarray}
Where $Z^t$ and $\overline{Z}$ denote the transpose and the conjugate of the row vector $Z=(z,w)$, respectively, $I_{d_0}$ denotes the identity matrix of order $d_0$, and symbols
$$
\frac{\partial}{\partial z^t}:=\left(\frac{\partial}{\partial z_1},\cdots,\frac{\partial}{\partial z_d}\right)^t,\;
\frac{\partial}{\partial \bar{z}}:=\left(\frac{\partial}{\partial \bar{z_1}},\cdots,\frac{\partial}{\partial \bar{z_d}}\right),\;
\frac{\partial^2}{\partial z^t\partial \bar z}:=\left(\frac{\partial^2}{\partial z_i\partial \bar z_j}\right).$$
 }\end{Lemma}

\begin{proof}[Proof]
By direct computation,  it follows that \eqref{e4.4} and \eqref{e4.6}.

Using \eqref{e4.4} and \eqref{e4.6}, we have
\begin{equation*}\label{e4.13}
T_4^{-1}=\frac{r^2}{F'}I_{d_0}+(\frac{1}{F''}-\frac{1}{F'})\overline{w}^tw,
\end{equation*}

\begin{equation*}\label{e4.14}
T_2T_4^{-1}T_3=F''\frac{\partial\phi}{\partial z^t}\frac{\partial\phi}{\partial \bar{z}}
\end{equation*}
and
\begin{equation*}\label{e4.15}
T_1-T_2T_4^{-1}T_3=(1+F')\frac{\partial^2\phi}{\partial z^t\partial \bar z}.
\end{equation*}
So we get
\begin{eqnarray*}
   \det T &=&\det(T_1-T_2T_4^{-1}T_3)\det T_4  \\
   &=& \frac{1}{r^{2d_0}}(F')^{d_0-1}F''(1+ F')^d\det(\frac{\partial^2\phi}{\partial z^t\partial \bar z}),
\end{eqnarray*}
\begin{equation*}\label{e4.16}
(T_1-T_2T_4^{-1}T_3)^{-1}=\frac{1}{1+F'}\left(\frac{\partial^2\phi}{\partial z^t\partial \bar z}\right)^{-1},
\end{equation*}
\begin{equation*}\label{e4.17}
-(T_1-T_2T_4^{-1}T_3)^{-1}T_2T_4^{-1}=-\frac{1}{1+F'}\left(\frac{\partial^2\phi}{\partial z^t\partial \bar z}\right)^{-1}\frac{\partial\phi}{\partial z^t}w,
\end{equation*}
\begin{equation*}\label{e4.18}
-T_4^{-1}T_3 (T_1-T_2T_4^{-1}T_3)^{-1}=-\frac{1}{1+F'}\bar{w}^t\frac{\partial\phi}{\partial \bar{z}}\left(\frac{\partial^2\phi}{\partial z^t\partial \bar z}\right)^{-1}
\end{equation*}
and
\begin{eqnarray*}
\nonumber       & & T_4^{-1}+T_4^{-1}T_3 (T_1-T_2T_4^{-1}T_3)^{-1}T_2T_4^{-1} \\
\label{e4.19}      &=& \frac{r^2}{F'}I_{d_0}+\left(\frac{1}{F''}-\frac{1}{F'}\right)\overline{w}^tw+\frac{1}{1+F'}\bar{w}^t\frac{\partial\phi}{\partial \bar{z}}\left(\frac{\partial^2\phi}{\partial z^t\partial \bar z}\right)^{-1}\frac{\partial\phi}{\partial z^t}w.
\end{eqnarray*}

Since
\begin{equation*}\label{e4.20}
  T^{-1}=\left(
           \begin{array}{ll}
             (T_1-T_2T_4^{-1}T_3)^{-1} & -(T_1-T_2T_4^{-1}T_3)^{-1}T_2T_4^{-1} \\
             -T_4^{-1}T_3 (T_1-T_2T_4^{-1}T_3)^{-1} & T_4^{-1}+T_4^{-1}T_3 (T_1-T_2T_4^{-1}T_3)^{-1}T_2T_4^{-1} \\
           \end{array}
         \right),
\end{equation*}
we have \eqref{e4.9}, \eqref{e4.10}, \eqref{e4.11} and \eqref{e4.12}.
\end{proof}

\begin{Remark}\label{Re:4.1}
Under assumptions of Lemma \ref{Le:4.1},
from
\begin{eqnarray*}
   && \left(
    \begin{array}{cc}
      I_d & -T_2T_4^{-1} \\
      0 & I_{d_0} \\
    \end{array}
  \right)\left(
           \begin{array}{cc}
             T_1 & T_2 \\
             T_3 & T_4 \\
           \end{array}
         \right)\left(
                  \begin{array}{cc}
                    I_d & 0 \\
                    -T^{-1}_4T_3 & I_{d_0} \\
                  \end{array}
                \right)  \\
   &=& \left(
                          \begin{array}{cc}
                           (1+F')\frac{\partial^2\phi}{\partial z^t\partial \bar z}  & 0 \\
                            0 & \frac{F'}{r^2}I_{d_0}+\frac{F''-F'}{r^4}\bar{w}^tw \\
                          \end{array}
                        \right)
\end{eqnarray*}
 we obtain that $\Phi_F$ is a K\"{a}hler potential on
 $$ M^{*}=\left\{(z,w)\in \Omega\times\mathbb{C}^{d_0}: 0<\|w\|^2<e^{-\phi(z)}\right\}$$
if and only if

 \begin{equation}\label{ap2.1}
   (1+F')\frac{\partial^2\phi}{\partial z^t\partial \bar z}>0
 \end{equation}
 and
\begin{equation}\label{ap2.2}
    \frac{F'}{r^2}I_{d_0}+\frac{F''-F'}{r^4}\bar{w}^tw>0.
\end{equation}

For the case of $d_0=1$,  \eqref{ap2.1} and \eqref{ap2.2} are equivalent to
\begin{equation}\label{ap2.3}
 1+F'(t)>0,\;\frac{F''(t)}{e^t}>0,\;t\in (-\infty,0),
\end{equation}
respectively.

For the case of $d_0>1$,  by the eigenvalues of the matrix $ \frac{F'}{r^2}I_{d_0}+\frac{F''-F'}{r^4}\bar{w}^tw$ are
\begin{equation*}
  \frac{F'(t)}{r^2},\ldots,\frac{F'(t)}{r^2},\frac{F''(t)}{r^2},
\end{equation*}
so \eqref{ap2.1} and \eqref{ap2.2} are equivalent to
\begin{equation}\label{ap2.4}
\frac{F'(t)}{e^t}>0,\;\frac{F''(t)}{e^t}>0,\;t\in (-\infty,0),
\end{equation}
respectively.
\end{Remark}

\begin{Lemma}\label{Le:4.2}{
Let $\phi$ be a globally defined real K\"{a}hler potential on a domain $\Omega$,
$$\Phi_F(z,w)=\phi(z)+F(\phi(z)+\ln \|w\|^2)$$
and
$$ M=\left\{(z,w)\in \Omega\times\mathbb{C}^{d_0}: \|w\|^2<e^{-\phi(z)}\right\}.$$

For given $(z_0,w_0)\in  M$, let
\begin{equation*}\label{e4.21}
  \widetilde{\phi}(u):=\phi(u+z_0)-\phi(z_0)-\frac{\partial\phi}{\partial z}(z_0)u^t-\frac{\partial\phi}{\partial \bar{z}}(z_0)\overline{u}^t,
\end{equation*}
\begin{equation*}\label{e4.22}
 \widetilde{\Omega}:=\{u\in \mathbb{C}^d: u+z_0\in \Omega\},
\end{equation*}

\begin{equation*}\label{e4.23}
  \widetilde{M}:=\left\{(u,v)\in \widetilde{\Omega}\times \mathbb{C}^{d_0}:\|v\|^2<e^{-\widetilde{\phi}(u)}\right\}
\end{equation*}
and
\begin{equation*}\label{e4.24}
 \widetilde{\Phi}_F(u,v):=\widetilde{\phi}(u)+F(\widetilde{\phi}(u)+\ln \|v\|^2).
\end{equation*}

Define the holomorphic mapping $\Upsilon$
\begin{equation*}\label{e4.25}
\begin{array}{rll}
 \Upsilon:  M & \rightarrow & \widetilde{M}, \\
  (z,w) & \mapsto & (u,v)=\left(z-z_0,e^{\frac{1}{2}\phi(z_0)+\frac{\partial\phi}{\partial z}(z_0)(z-z_0)^t}w\right).
\end{array}
\end{equation*}
Then
$$\phi(z)+\ln \|w\|^2=\widetilde{\phi}(u)+\ln \|v\|^2$$
and
\begin{equation*}\label{e4.26}
  \partial\bar{\partial}\Phi_F=\partial\bar{\partial}(\widetilde{\Phi}_F\circ \Upsilon).
\end{equation*}
 }\end{Lemma}

\begin{proof}[Proof]
The proof is trivial, we omit it.
\end{proof}

By Lemma \ref{Le:4.2}, the scalar curvature, the Laplace, the squared norm of the curvature tensor and the squared norm of the
Ricci curvature at $(z_0,w_0)$ associated with the K\"{a}hler potential $\Phi_F$  on the domain $ M$ are equal to the scalar curvature, the Laplace, the squared norm of the curvature tensor and the squared norm of the Ricci curvature at $(0,e^{\frac{1}{2}\phi(z_0)}w_0)$ associated with the K\"{a}hler potential $\widetilde{\Phi}_F$  on the domain $\widetilde{M}$, respectively. For convenience, the following we assume that  $0\in \Omega$ and
\begin{equation*}\label{e4.17}
  \phi(0)=0,\;\frac{\partial\phi}{\partial z^t}(0)=0,\;\frac{\partial\phi}{\partial \bar{z}}(0)=0.
\end{equation*}

\begin{Theorem}\label{Th:5.1}{Assume that $\phi$
is a globally defined  K\"{a}hler potential on a domain $\Omega\subset \mathbb{C}^d$. Let $g_{\phi}$ be a K\"{a}hler  metric  on the domain $\Omega$ associated with the K\"{a}hler form $\omega_{\phi}=\frac{\sqrt{-1}}{2\pi}\partial\overline{\partial}\phi$, and $g_F$ be a K\"{a}hler  metric  on the domain $ M$ associated with the K\"{a}hler form $\omega_F=\frac{\sqrt{-1}}{2\pi}\partial\overline{\partial}\Phi_F$, where
$$\Phi_F(z,w)=\phi(z)+F(t)$$
with $t=\phi(z)+\ln\|w\|^2$ is a K\"{a}hler potential on a Hartogs domain
$$ M=\left\{(z,w)\in \Omega\times\mathbb{C}^{d_0}: \|w\|^2<e^{-\phi(z)}\right\}.$$

Set
\begin{equation}\label{e5.02}
  G=(d_0-1)\ln F'+\ln F''+d\ln(1+F'),
\end{equation}

\begin{equation}\label{e5.01}
 \psi_1=-d\frac{G'}{1+F'}+(d_0-1)\frac{d_0-G'}{F'}-\frac{G''}{F''}
\end{equation}
and
\begin{equation}\label{e5.03}
  \psi_2=d\left(\frac{G'}{1+F'}\right)^2+(d_0-1)\left(\frac{G'-d_0}{F'}\right)^2+\left(\frac{G''}{F''}\right)^2.
\end{equation}

Then
\begin{equation}\label{e5.1}
  k_{g_F}=\frac{1}{1+F'}k_{g_{\phi}}+\psi_1,
\end{equation}
\begin{equation}\label{e5.2}
  |\mathrm{Ric}_{g_F}|^2=\frac{1}{(1+F')^2}|\mathrm{Ric}_{g_{\phi}}|^2-\frac{2G'}{(1+F')^2} k_{g_{\phi}}+\psi_2,
\end{equation}
\begin{equation}\label{e5.3}
  \triangle_{g_F} k_{g_F}=\frac{1}{(1+F')^2}(\triangle_{g_{\phi}}k_{g_{\phi}}) +\frac{1}{F''}\frac{\partial^2k_{g_F}}{\partial t^2}+\left(\frac{d}{1+F'}+\frac{d_0-1}{F'}\right) \frac{\partial k_{g_F}}{\partial t},
\end{equation}
where $k_{g_{\phi}}$, $\Delta_{g_{\phi}}$ and $\mathrm{Ric}_{g_{\phi}}$ denote the scalar curvature, the Laplace and the Ricci curvature on the domain $\Omega$ with respect to the metric $g_{\phi}$, respectively.
}\end{Theorem}

\begin{proof}[Proof]
Without loss of generality, let $0\in \Omega$, there exists local coordinates $(z_1, \cdots, z_d)$ on a neighborhood of a point $0$ such that the K\"{a}hler potential $\phi$ on the domain $\Omega$ is given locally by
\begin{equation}\label{e5.5}
  \phi(z)=\|z\|^2+\sum_{i,j,k,l=1}^dc_{i\bar{j}k\bar{l}}\; z_i\bar{z_j}z_k\bar{z_l}+O(\|z\|^5).
\end{equation}
By Lemma \ref{Le:4.2}, to compute $k_{g_F}(z,w)$, $|\mathrm{Ric}_{g_F}|^2(z,w)$ and
$\triangle_{g_F} k_{g_F}(z,w) $, we only need to calculate $k_{g_F}(0,w)$,
$|\mathrm{Ric}_{g_F}|^2(0,w)$ and $\triangle_{g_F} k_{g_F}(0,w) $.

Using Lemma \ref{Le:4.1}, we get
\begin{equation}\label{e5.7}
  T(0,w)=\left(
           \begin{array}{ll}
            (1+F') I_d  & 0 \\
             0 & \frac{F'}{r^2}I_{d_0}+\frac{F''-F'}{r^4}\bar{w}^tw \\
           \end{array}
         \right),
\end{equation}
\begin{equation}\label{e5.8}
  T^{-1}(0,w)=(T^{\bar{i}j})(0,w)=\left(
                                        \begin{array}{ll}
                                         \frac{1}{1+F'} I_d & 0 \\
                                          0 & \frac{r^2}{F'}I_{d_0}+\left(\frac{1}{F''}-\frac{1}{F'}\right)\overline{w}^tw \\
                                        \end{array}
                                      \right),
\end{equation}
and
\begin{equation}\label{e5.9}
\ln\det T=G-d_0\ln \|w\|^2+\ln\left(\det\left(\frac{\partial^2\phi}{\partial z^t\partial \bar z}\right)\right).
\end{equation}

Let
\begin{equation*}\label{e5.10}
 \mathrm{Ric}_{g_{\phi}}=-\frac{\partial^2}{\partial z^t\partial\bar{z}}\ln\left(\det\left(\frac{\partial^2\phi}{\partial z^t\partial \bar z}\right)\right).
\end{equation*}

From \eqref{e5.5} and \eqref{e5.9}, it follows that
\begin{eqnarray*}
\nonumber   & &  \mathrm{Ric}_{g_F}(0,w):=-\frac{\partial^2\ln\det T}{\partial Z^t\partial\overline{Z}}(0,w) \\
\label{e5.11}   &=&-\left(
                                                                                \begin{array}{ll}
                                                                                  G'I_d- \mathrm{Ric}_{g_{\phi}}(0) & 0  \\
                                                                                  0 & \frac{G'-d_0}{r^2}I_{d_0}+\frac{G''-G'+d_0}{r^4}\bar{w}^tw \\
                                                                                \end{array}
                                                                              \right),
\end{eqnarray*}
which implies that
\begin{eqnarray*}
   & & \left(T^{-1}\mathrm{Ric}_{g_F}\right)(0,w) \\
   &=& -\left(
                             \begin{array}{ll}
                               \frac{G'}{1+F'}I_d-\frac{1}{1+F'} \mathrm{Ric}_{g_{\phi}}(0) & 0 \\
                               0 & \frac{G'-d_0}{F'}I_{d_0}+\frac{d_0F''+F'G''-F''G'}{r^2F'F''}\bar{w}^tw \\
                             \end{array}
                           \right).
\end{eqnarray*}
So
\begin{equation*}\label{e5.12}
  k_{g_F}(0,w)=\mathrm{Tr}\left(T^{-1}\mathrm{Ric}_{g_F}\right)(0,w)=\frac{1}{1+F'}k_{g_{\phi}}+\psi_1(t)
\end{equation*}
and
\begin{eqnarray*}
\nonumber |\mathrm{Ric}_{g_F}|^2(0,w)  &=& \mathrm{Tr}\left(T^{-1}\mathrm{Ric}_{g_F}\;T^{-1}\mathrm{Ric}_{g_F}\right)(0,w) \\
\nonumber   &=&\frac{1}{(1+F')^2}|\mathrm{Ric}_{g_{\phi}}|^2(0)-2\frac{G'}{(1+F')^2} k_{g_{\phi}}(0)  \\
\label{e5.13}   & &+d\left(\frac{G'}{1+F'}\right)^2+(d_0-1)\left(\frac{G'-d_0}{F'}\right)^2
+\left(\frac{G''}{F''}\right)^2,
\end{eqnarray*}
where
\begin{equation*}
  k_{g_{\phi}}(0)=\mathrm{Tr}\left(\mathrm{Ric}_{g_{\phi}}\right)(0),\; |\mathrm{Ric}_{g_{\phi}}|^2(0)=\mathrm{Tr}\left(\mathrm{Ric}_{g_{\phi}}\mathrm{Ric}_{g_{\phi}}\right)(0).
\end{equation*}
Then, we obtain \eqref{e5.1} and \eqref{e5.2}.

Applying \eqref{e5.1} and \eqref{e5.5}, we obtain
\begin{eqnarray*}
   & &  \frac{\partial^2k_{g_F}}{\partial Z^t\partial\overline{Z}}(0,w) \\
   &=& \left(
                                                           \begin{array}{ll}
                                                             \frac{\partial k_{g_F}}{\partial t}I_d+\frac{1}{1+F'} \frac{\partial^2k_{g_{\phi}}}{\partial z^t\partial\bar{z}}(0)& -\frac{F''}{r^2(1+F')^2}\frac{\partial k_{g_{\phi}}}{\partial z^t}(0)w \\
                                                             -\frac{F''}{r^2(1+F')^2}\bar{w}^t\frac{\partial k_{g_{\phi}}}{\partial \bar{z}}(0) & \frac{1}{r^2}\frac{\partial k_{g_F}}{\partial t}I_{d_0}+\frac{1}{r^4}(\frac{\partial^2 k_{g_F}}{\partial t^2}-\frac{\partial k_{g_F}}{\partial t})\bar{w}^tw \\
                                                           \end{array}
                                                         \right),
\end{eqnarray*}
here $Z=(z,w)$. Therefore
\begin{eqnarray}
\nonumber   & &(\triangle_{g_F} k_{g_F})(0,w)\\
\nonumber   &=& \mathrm{Tr}\left(T^{-1}  \frac{\partial^2k_{g_F}}{\partial Z^t\partial\overline{Z}}\right)(0,w) \\
\nonumber   &=& \frac{d}{1+F'}\frac{\partial k_{g_F}}{\partial t}+\frac{1}{(1+F')^2}(\triangle_{g_{\phi}}k_{g_{\phi}})(0)+\frac{d_0-1}{F'}\frac{\partial k_{g_F}}{\partial t}+\frac{1}{F''}\frac{\partial^2 k_{g_F}}{\partial t^2},
\end{eqnarray}
thus we get \eqref{e5.3}.
\end{proof}

Let
\begin{equation*}
  x=F'(t),\;\varphi(x)=F''(t).
\end{equation*}
In Lemma \ref{Le:5.2} below, we obtain expressions in $\varphi(x)$ for the scalar
curvature $k_{g_F}$, the squared norm $|\mathrm{Ric}_{g_F}|^2$ of the Ricci curvature
tensor, and the Laplace $\triangle_{g_F} k_{g_F}$ of the scalar curvature.

\begin{Lemma}\label{Le:5.2}{Under assumptions of Theorem \ref{Th:5.1}, let
$$x=F'(t),\;\varphi(x)=F''(t)\;,$$
\begin{equation}\label{e5.60}
  \sigma=\frac{\left((1+x)^dx^{d_0-1}\varphi\right)'}{(1+x)^dx^{d_0-1}}
\end{equation}
and
\begin{equation}\label{e5.61}
 \chi=\frac{d_0(d_0-1)}{x}-\frac{\left((1+x)^dx^{d_0-1}\varphi\right)''}{(1+x)^dx^{d_0-1}}.
\end{equation}

Then
\begin{equation}\label{e5.54}
  k_{g_F}=\frac{1}{1+x}k_{g_{\phi}}+\frac{d_0(d_0-1)}{x}-\frac{\left((1+x)^dx^{d_0-1}\varphi\right)''}{(1+x)^dx^{d_0-1}},
\end{equation}
\begin{eqnarray}
\nonumber |\mathrm{Ric}_{g_F}|^2  &=& \frac{1}{(1+x)^2}|\mathrm{Ric}_{g_{\phi}}|^2-2\frac{\sigma}{(1+x)^2} k_{g_{\phi}}+(\sigma')^2 \\
\label{e5.55}   & & +d\left(\frac{\sigma}{1+x}\right)^2+(d_0-1)\left(\frac{\sigma-d_0}{x}\right)^2
\end{eqnarray}
and
\begin{eqnarray}
\nonumber  & & \triangle_{g_F} k_{g_F}\\
\label{e5.56}   &=& \frac{1}{(1+x)^2}(\triangle_{g_{\phi}}k_{g_{\phi}})- \frac{\left(\varphi(1+x)^{d-2}x^{d_0-1}\right)'}{(1+x)^dx^{d_0-1}}k_{g_{\phi}}+\frac{\left(\varphi\chi'(1+x)^dx^{d_0-1}\right)'}{(1+x)^dx^{d_0-1}},
\end{eqnarray}
where $k_{g_{\phi}}$, $\Delta_{g_{\phi}}$ and $\mathrm{Ric}_{g_{\phi}}$  denote the scalar curvature, the Laplace and the Ricci curvature  on the domain $\Omega$ with respect to the metric $g_{\phi}$, respectively.
}\end{Lemma}

\begin{proof}[Proof]
Using
\begin{equation*}
  x=F'(t),\;\varphi(x)=F''(t),
\end{equation*}
we give
\begin{equation*}
  F'''(t)=\varphi'(x)\frac{dx}{dt}=\varphi'(x)F''(t)=\varphi'(x)\varphi(x).
\end{equation*}
Then
\begin{eqnarray*}
 G' &=& (d_0-1)\frac{F''}{F'}+\frac{F'''}{F''}+d\frac{F''}{1+F'}\\
    &=&\frac{\left((1+x)^dx^{d_0-1}\varphi\right)'}{(1+x)^dx^{d_0-1}}.
\end{eqnarray*}
Let
\begin{equation*}
  \sigma=\frac{\left((1+x)^dx^{d_0-1}\varphi\right)'}{(1+x)^dx^{d_0-1}},
\end{equation*}
we have
\begin{equation}\label{ee5.61}
  G'(t)=\sigma(x)
\end{equation}
and
\begin{equation}\label{ee5.62}
   G''(t)=\sigma'(x)\frac{dx}{dt}=\sigma'(x)\varphi(x).
\end{equation}

By \eqref{ee5.61} and \eqref{ee5.62}, we obtain
\begin{eqnarray*}
\psi_1   &=&-d\frac{G'}{1+F'}+(d_0-1)\frac{d_0-G'}{F'}-\frac{G''}{F''} \\
   &=&  \frac{d_0(d_0-1)}{x}-\frac{\left((1+x)^dx^{d_0-1}\sigma\right)'}{(1+x)^dx^{d_0-1}},
\end{eqnarray*}
which implies \eqref{e5.54}.

From \eqref{ee5.61} and \eqref{ee5.62}, we also give
\begin{eqnarray*}
\psi_2   &=& d\left(\frac{G'}{1+F'}\right)^2+(d_0-1)\left(\frac{G'-d_0}{F'}\right)^2+\left(\frac{G''}{F''}\right)^2 \\
   &=& d\frac{\sigma^2}{(1+x)^2}+(d_0-1)\left(\frac{\sigma-d_0}{x}\right)^2+(\sigma')^2,
\end{eqnarray*}
thus
\begin{eqnarray*}
  & &|\mathrm{Ric}_{g_F}|^2\\
  &=&\frac{1}{(1+F')^2}|\mathrm{Ric}_{g_{\phi}}|^2-\frac{2G'}{(1+F')^2} k_{g_{\phi}}+\psi_2  \\
  &=& \frac{1}{(1+x)^2}|\mathrm{Ric}_{g_{\phi}}|^2-2\frac{\sigma}{(1+x)^2} k_{g_{\phi}}+(\sigma')^2+d\frac{\sigma^2}{(1+x)^2}+(d_0-1)\left(\frac{\sigma-d_0}{x}\right)^2.
\end{eqnarray*}

Since
\begin{equation*}
  \frac{\partial k_{g_F}}{\partial t}=\frac{\partial}{\partial x}\left(\frac{1}{1+x}k_{g_{\phi}}+\chi\right)\frac{dx}{dt}= -\frac{\varphi}{(1+x)^2}k_{g_{\phi}}+\varphi\chi'
\end{equation*}
and
\begin{equation*}
 \frac{\partial^2 k_{g_F}}{\partial t^2}=\left(-\left(\frac{\varphi}{(1+x)^2}\right)'k_{g_{\phi}}+(\varphi\chi')'\right)\varphi,
\end{equation*}
then
\begin{eqnarray*}
  & & \triangle_{g_F} k_{g_F}\\
  &=& \frac{1}{(1+F')^2}(\triangle_{g_{\phi}}k_{g_{\phi}}) +\frac{1}{F''}\frac{\partial^2k_{g_F}}{\partial t^2}+\left(\frac{d}{1+F'}+\frac{d_0-1}{F'}\right) \frac{\partial k_{g_F}}{\partial t}\\
   &=&\frac{1}{(1+x)^2}(\triangle_{g_{\phi}}k_{g_{\phi}})-\left(\frac{dx+(d_0-1)(1+x)}{x(1+x)^3}\varphi+\left(\frac{\varphi}{(1+x)^2}\right)'\right)k_{g_{\phi}}    \\
   & &+(\varphi\chi')'+\frac{dx+(d_0-1)(1+x)}{x(1+x)}\varphi\chi'.
\end{eqnarray*}
\end{proof}

In order to obtain the coefficient $\mathbf{a}_2$ of  the Bergman function expansion  for $(M,g_F)$, we give a key Lemma \ref{Le:5.3} which  gives an explicit expression of the squared norm $|R_{g_F}|^2$ of the curvature tensor of the metric $g_F$.

\begin{Lemma}\label{Le:5.3}
Under the situation of Theorem \ref{Th:5.1}, let
$$t=\phi(z)+\ln\|w\|^2,\;x=F'(t),\;\varphi(x)=F''(t)\;.$$
Then
\begin{eqnarray}
\nonumber    & &|R_{g_F}|^2 \\
\nonumber     &=& \frac{1}{(1+x)^2}|R_{g_{\phi}}|^2-\frac{4\varphi}{(1+x)^3}k_{g_{\phi}}+2d(d+1)\frac{\varphi^2}{(1+x)^4}+4d\left(\left(\frac{\varphi}{1+x}\right)'\right)^2 \\
\label{e8.1}   & & +\left(\varphi''\right)^2+(d_0-1)\left\{4d \left(\frac{\varphi}{x(1+x)}\right)^2+4\left(\left(\frac{\varphi}{x}\right)'\right)^2+2d_0\left(\frac{\varphi-x}{x^2}\right)^2\right\},
\end{eqnarray}
where $k_{g_{\phi}}$ and $R_{g_{\phi}}$ denote the scalar curvature and the curvature tensor on the domain $\Omega$ with respect to the metric $g_{\phi}$, respectively.
\end{Lemma}

\begin{proof}[Proof]

Let $Z=(z,w)$, $T=\frac{\partial^2\Phi_F}{\partial Z^t\partial\overline{Z}}$,
\begin{equation}\label{e8.2}
  R_{g_F}\equiv(R_{i\bar{j}}):=-\partial\bar{\partial}T+(\partial T)T^{-1}\wedge(\bar{\partial}T)
\end{equation}
and
\begin{equation}\label{e8.3}
 R_{i\bar{j}}=\sum_{k,l=1}^nR_{i\bar{j}k\bar{l}}dZ_k\wedge d\overline{Z_l}.
\end{equation}

Since the metric $g_F$ is invariant under  transformations
\begin{equation*}
  (z,w)\in M\longmapsto (z,wU)\in M, \;U\in \mathcal{U}(d_0),
\end{equation*}
 where $\mathcal{U}(d_0)$ indicates the unitary group of order $d_0$, we only need to compute the
squared norm of the curvature tensor of the metric $g_F$ at $(z,w)=(0,w_1,0,\cdots,0)$.

Let $\phi$ be given locally by
\begin{equation*}
  \phi(z)=\|z\|^2+\sum_{i,j,k,l=1}^dc_{i\bar{j}k\bar{l}}\; z_i\bar{z_j}z_k\bar{z_l}+O(\|z\|^5).
\end{equation*}
Then
\begin{equation*}
  \partial\phi(0)=0,\;\bar{\partial}\phi(0)=0,\;\partial\bar{\partial}\phi(0)=\sum_{k=1}^ddz_k\wedge d\bar{z_k},
  \; (\partial\frac{\partial\phi}{\partial z^t})(0)=0
\end{equation*}
and
\begin{equation*}
  (\partial\frac{\partial\phi}{\partial \bar{z}})(0)=dz,\; (\bar{\partial}\frac{\partial\phi}{\partial z^t})(0)=(d\bar{z})^t,\;
  (\bar{\partial}\frac{\partial\phi}{\partial \bar{z}})(0)=0,\;
  \partial\bar{\partial}(\frac{\partial\phi}{\partial z^t})(0)=0,\;\partial\bar{\partial}(\frac{\partial\phi}{\partial \bar{z}})(0)=0.
\end{equation*}

The following we set $z=0,\mathbf{e_1}=(1,0,\ldots,0)\in \mathbb{C}^{d_0},w=w_1\mathbf{e_1}$, $du=(dw_2,dw_3,\ldots,dw_{d_0}),dw=(dw_1,du)$, we get
\begin{equation*}
  \partial t=\frac{1}{r^2}\overline{w_1}dw_1,\;\bar{\partial} t=\frac{1}{r^2}w_1d\overline{w_1},\;\partial\bar{\partial}t=dz\wedge (d\bar{z})^t+\frac{1}{r^2}du\wedge (d\bar{u})^t,
\end{equation*}
where $r^2=|w_1|^2$.

By using Lemma \ref{Le:4.1} and
\begin{equation*}
  x=F'(t),\;\varphi(x)=F''(t),\;F'''(t)=\varphi(x)\varphi'(x),\;F^{(4)}(t)=\varphi^2(x)\varphi''(x)+\varphi(x)(\varphi'(x))^2,
\end{equation*}
we have
\begin{eqnarray}
\nonumber  \partial T &=& \begin{array}{c}
     \scriptstyle  d \\
     \scriptstyle  d_0
     \end{array}
     \overset{\begin{array}{cc}
           \scriptstyle  d &\;\;\scriptstyle d_0
           \end{array}}
     {\left(\begin{array}{c|c}
             (\partial T)_1 & (\partial T)_2 \\
             \hline
             (\partial T)_3 & (\partial T)_4
           \end{array}
     \right)} \\
\nonumber   &=& \overline{w_1}\left(
                    \begin{array}{ll}
                      \frac{F''}{r^2}dw_1I_d & 0 \\
                     \frac{F''}{r^2}\mathbf{e_1}^tdz &  \widetilde{(\partial T)_4}  \\
                    \end{array}
                  \right) \\
\label{e8.4}   &=&\overline{w_1}\left(
                    \begin{array}{ll}
                     \frac{\varphi}{r^2}dw_1I_d & 0 \\
                     \frac{\varphi}{r^2}\mathbf{e_1}^tdz & \widetilde{(\partial T)_4}
                    \end{array}
                  \right),
\end{eqnarray}

\begin{eqnarray}
\nonumber  \widetilde{(\partial T)_4}  &=& \begin{array}{c}
     \scriptstyle  1 \\
     \scriptstyle  d_0-1
     \end{array}
     \overset{\begin{array}{cc}
           \scriptstyle  1 &\;\;\;\scriptstyle d_0-1
           \end{array}}
     {\left(\begin{array}{c|c}
             (\partial T)_{41} & (\partial T)_{42} \\
             \hline
             (\partial T)_{43} & (\partial T)_{44}
           \end{array}
     \right)} \\
\nonumber   &=& \frac{F''-F'}{r^4}\left(
                                                                                         \begin{array}{ll}
                                                                                           \frac{F'''-F''}{F''-F'}dw_1 & du \\
                                                                                           0 & dw_1I_{d_0-1} \\
                                                                                         \end{array}
                                                                                       \right) \\
\label{e8.4-1}   &=& \frac{\varphi-x}{r^4}\left(
                                                                                         \begin{array}{ll}
                                                                                           \frac{\varphi'\varphi-\varphi}{\varphi-x}dw_1 & du \\
                                                                                           0 & dw_1I_{d_0-1} \\
                                                                                         \end{array}
                                                                                       \right),
\end{eqnarray}

\begin{eqnarray}
\nonumber \overline{\partial} T  &=& \begin{array}{c}
     \scriptstyle  d \\
     \scriptstyle  d_0
     \end{array}
     \overset{\begin{array}{cc}
           \scriptstyle  d &\;\;\scriptstyle d_0
           \end{array}}
     {\left(\begin{array}{c|c}
             (\overline{\partial} T)_1 & (\overline{\partial} T)_2 \\
             \hline
             (\overline{\partial} T)_3 & (\overline{\partial} T)_4
           \end{array}
     \right)}  \\
\nonumber   &=& w_1\left(
                    \begin{array}{ll}
                      \frac{F''}{r^2}d\overline{w_1}I_d & \frac{F''}{r^2}(d\bar{z})^t\mathbf{e_1} \\
                      0 & \widetilde{(\overline{\partial} T)_4}  \\
                    \end{array}
                  \right) \\
\label{e8.5}   &=& w_1\left(
                    \begin{array}{ll}
                     \frac{\varphi}{r^2}d\overline{w_1}I_d &\frac{\varphi}{r^2}(d\bar{z})^t\mathbf{e_1} \\
                     0 & \widetilde{ (\overline{\partial} T)_4 }                \\
                    \end{array}
                  \right),
\end{eqnarray}

\begin{eqnarray}
\nonumber  \widetilde{(\overline{\partial} T)_4}  &=& \begin{array}{c}
     \scriptstyle  1 \\
     \scriptstyle  d_0-1
     \end{array}
     \overset{\begin{array}{cc}
           \scriptstyle  1 &\;\;\;\scriptstyle d_0-1
           \end{array}}
     {\left(\begin{array}{c|c}
             (\overline{\partial} T)_{41} & (\overline{\partial} T)_{42} \\
             \hline
             (\overline{\partial} T)_{43} & (\overline{\partial} T)_{44}
           \end{array}
     \right)} \\
\nonumber   &=& \frac{F''-F'}{r^4}\left(
                                                                                         \begin{array}{ll}
                                                                                           \frac{F'''-F''}{F''-F'}d\overline{w_1} & 0 \\
                                                                                           (d\overline{u})^t & d\overline{w_1}I_{d_0-1} \\
                                                                                         \end{array}
                                                                                       \right) \\
\label{e8.5-1}   &=& \frac{\varphi-x}{r^4}\left(
                                                                                         \begin{array}{ll}
                                                                                           \frac{\varphi'\varphi-\varphi}{\varphi-x}d\overline{w_1} & 0 \\
                                                                                           (d\overline{u})^t & d\overline{w_1}I_{d_0-1} \\
                                                                                         \end{array}
                                                                                       \right)
\end{eqnarray}
and
\begin{equation}\label{e8.6}
  \partial\bar{\partial}T=\begin{array}{c}
     \scriptstyle  d \\
     \scriptstyle  d_0
     \end{array}
     \overset{\begin{array}{cc}
           \scriptstyle  d &\;\;\scriptstyle d_0
           \end{array}}
     {\left(\begin{array}{c|c}
             (\partial\overline{\partial} T)_1 & (\partial\overline{\partial} T)_2 \\
             \hline
             (\partial\overline{\partial} T)_3 & (\partial\overline{\partial} T)_4
           \end{array}
     \right)} ,
\end{equation}
where
\begin{eqnarray}
\nonumber   (\partial\bar{\partial}T)_1  &=& \left\{F'' dz\wedge (d\bar{z})^t
+\frac{F'''}{r^2}dw_1\wedge d\overline{w_1}+\frac{F''}{r^2}du\wedge(d\bar{u})^t\right\}I_d \\
\nonumber    & & +\left(1+F'\right)\partial\bar{\partial}(\frac{\partial^2\phi}{\partial z^t\partial\bar{z}})(0)-F''(d\bar{z})^t\wedge dz\\
\nonumber &=& \left\{\varphi dz\wedge (d\bar{z})^t
+\frac{\varphi\varphi'}{r^2}dw_1\wedge d\overline{w_1}+\frac{\varphi}{r^2}du\wedge (d\overline{u})^t\right\}I_d\\
\label{e8.7}    & &+\left(1+x\right)\partial\bar{\partial}(\frac{\partial^2\phi}{\partial z^t\partial\bar{z}})(0)-\varphi (d\bar{z})^t\wedge dz,
\end{eqnarray}
\begin{eqnarray}
\nonumber   (\partial\bar{\partial}T)_2 &=& -\left\{\frac{F''-F'''}{r^2}dw_1\wedge(d\bar{z})^t\mathbf{e_1}+\frac{F''}{r^2}(d\bar{z})^t\wedge dw\right\} \\
\label{e8.8}   &=& -\left\{\frac{\varphi(1-\varphi')}{r^2}dw_1\wedge(d\bar{z})^t\mathbf{e_1}+\frac{\varphi}{r^2}(d\bar{z})^t\wedge dw\right\},
\end{eqnarray}

\begin{eqnarray}
\nonumber   (\partial\bar{\partial}T)_3 &=& -\left\{\frac{F''-F'''}{r^2}\mathbf{e_1}^tdz\wedge d\overline{w_1}+\frac{F''}{r^2}(d\bar{w})^t\wedge dz\right\} \\
\label{e8.9}   &=& -\left\{\frac{\varphi(1-\varphi')}{r^2}\mathbf{e_1}^tdz\wedge d\overline{w_1}+\frac{\varphi}{r^2}(d\bar{w})^t\wedge dz\right\}
\end{eqnarray}
and
\begin{equation}\label{e8.10}
 (\partial\bar{\partial}T)_4=\begin{array}{c}
     \scriptstyle  1 \\
     \scriptstyle  d_0-1
     \end{array}
     \overset{\begin{array}{cc}
           \scriptstyle  1 &\;\;\;\scriptstyle d_0-1
           \end{array}}
     {\left(\begin{array}{c|c}
             (\partial\overline{\partial} T)_{41} & (\partial\overline{\partial} T)_{42} \\
             \hline
             (\partial\overline{\partial} T)_{43} & (\partial\overline{\partial} T)_{44}
           \end{array}
     \right)}
\end{equation}
with
\begin{eqnarray*}
& &(\partial\bar{\partial}T)_{41}\\
    &=&\frac{F'''}{r^2}dz\wedge (d\bar{z})^t+\frac{F^{(4)}-2F'''+F''}{r^4}dw_1\wedge d\overline{w_1}+\frac{F'''-2F''+F'}{r^4}du\wedge (d\overline{u})^t\\
   &=& \frac{\varphi\varphi'}{r^2}dz\wedge (d\bar{z})^t+\frac{\varphi^2\varphi''+\varphi(\varphi'-1)^2}{r^4}dw_1\wedge d\overline{w_1} +\frac{\varphi\varphi'-2\varphi+x}{r^4}du\wedge (d\overline{u})^t,
\end{eqnarray*}
\begin{equation*}
  (\partial\bar{\partial}T)_{42}=\frac{F'''-2F''+F'}{r^4}du\wedge d\overline{w_1}=\frac{\varphi\varphi'-2\varphi+x}{r^4}du\wedge d\overline{w_1},
\end{equation*}
\begin{equation*}
  (\partial\bar{\partial}T)_{43}=\frac{F'''-2F''+F'}{r^4}dw_1\wedge (d\overline{u})^t=\frac{\varphi\varphi'-2\varphi+x}{r^4}dw_1\wedge (d\overline{u})^t
\end{equation*}
and
\begin{eqnarray*}
   & &(\partial\bar{\partial}T)_{44}  \\
   &=& \left\{\frac{F''}{r^2}dz\wedge (d\bar{z})^t+\frac{F'''-2F''+F'}{r^4}dw_1\wedge d\overline{w_1}+\frac{F''-F'}{r^4}du\wedge (d\overline{u})^t\right\}I_{d_0-1} \\
   & &-\frac{F''-F'}{r^4}(d\overline{u})^t\wedge du\\
   &=&\left\{\frac{\varphi}{r^2}dz\wedge (d\bar{z})^t+\frac{\varphi\varphi'-2\varphi+x}{r^4}dw_1\wedge d\overline{w_1}+\frac{\varphi-x}{r^4}du\wedge (d\overline{u})^t\right\}I_{d_0-1}\\
   & &-\frac{\varphi-x}{r^4}(d\overline{u})^t\wedge du.
\end{eqnarray*}

According to \eqref{e5.8}, \eqref{e8.4}, \eqref{e8.4-1}, \eqref{e8.5} and \eqref{e8.5-1}, we get
\begin{equation}\label{e8.11}
  ((\partial T)T^{-1}\wedge(\bar{\partial}T))=\begin{array}{c}
     \scriptstyle  d \\
     \scriptstyle  d_0
     \end{array}
     \overset{\begin{array}{cc}
           \scriptstyle  d &\;\;\scriptstyle d_0
           \end{array}}
     {\left(\begin{array}{c|c}
           ((\partial T)T^{-1}\wedge(\bar{\partial}T))_1    & ((\partial T)T^{-1}\wedge(\bar{\partial}T))_2  \\
             \hline
           ((\partial T)T^{-1}\wedge(\bar{\partial}T))_3    &  ((\partial T)T^{-1}\wedge(\bar{\partial}T))_4
           \end{array}
     \right)} ,
\end{equation}
here
\begin{equation*}
  ((\partial T)T^{-1}\wedge(\bar{\partial}T))_1= \frac{\varphi^2}{1+x}\frac{1}{r^2}dw_1\wedge d\overline{w_1}I_d,
\end{equation*}
\begin{equation*}
  ((\partial T)T^{-1}\wedge(\bar{\partial}T))_2= \frac{\varphi^2}{1+x}\frac{1}{r^2}dw_1\wedge(d\bar{z})^t\mathbf{e_1},
\end{equation*}
\begin{equation*}
  ((\partial T)T^{-1}\wedge(\bar{\partial}T))_3= \frac{\varphi^2}{1+x}\frac{1}{r^2} \mathbf{e_1}^tdz\wedge d\overline{w_1}
\end{equation*}
and
\begin{eqnarray*}
   & &  ((\partial T)T^{-1}\wedge(\bar{\partial}T))_4 \\
  &=&\left(
       \begin{array}{ll}
        \frac{\varphi^2}{1+x}\frac{1}{r^2}dz\wedge (d\bar{z})^t+\frac{\varphi(1-\varphi')^2}{r^4}dw_1\wedge d\overline{w_1}+\frac{(\varphi-x)^2}{xr^4}du\wedge (d\overline{u})^t  & \frac{(\varphi-x)^2}{xr^4}du\wedge d\overline{w_1}\\
       \frac{(\varphi-x)^2}{xr^4}dw_1\wedge (d\overline{u})^t   & \frac{(\varphi-x)^2}{xr^4}dw_1\wedge d\overline{w_1}I_{d_0-1} \\
       \end{array}
     \right).
\end{eqnarray*}

Now \eqref{e8.6} and \eqref{e8.11} give
\begin{equation}\label{e8.12}
   R_{g_F}=\begin{array}{c}
     \scriptstyle  d \\
     \scriptstyle  d_0
     \end{array}
     \overset{\begin{array}{cc}
           \scriptstyle  d &\;\;\scriptstyle d_0
           \end{array}}
     {\left(\begin{array}{c|c}
           (R_{g_F})_1    & (R_{g_F})_2  \\
             \hline
           (R_{g_F})_3    &  (R_{g_F})_4
           \end{array}
     \right)},
\end{equation}
where
\begin{eqnarray}
\nonumber     & & (R_{g_F})_{1}\\
\nonumber     &=&\left\{-\varphi dz\wedge (d\bar{z})^t-\frac{(1+x)\varphi}{r^2}\left(\frac{\varphi}{1+x}\right)'dw_1\wedge d\overline{w_1}-\frac{\varphi}{r^2}du\wedge (d\overline{u})^t\right\}I_d\\
\label{e8.13} & &-(1+x)\partial\bar{\partial}(\frac{\partial^2\phi}{\partial z^t\partial\bar{z}})(0)+\varphi (d\bar{z})^t\wedge dz,
\end{eqnarray}
\begin{eqnarray}
\nonumber   & & (R_{g_F})_{2}\\
\nonumber   &=& -\frac{(1+x)\varphi}{r^2}\left(\left(\frac{\varphi}{1+x}\right)'-\frac{1}{1+x}\right) dw_1\wedge(d\bar{z})^t\mathbf{e}_1+\frac{\varphi}{r^2}(d\bar{z})^t\wedge dw\\
\label{e8.14}   &=&-\left(
                     \begin{array}{cc}
                      \frac{(1+x)\varphi}{r^2}\left(\frac{\varphi}{1+x}\right)' dw_1\wedge d\overline{z_1} & \frac{\varphi}{r^2}du\wedge d\overline{z_1} \\
                       \vdots & \vdots \\
                      \frac{(1+x)\varphi}{r^2}\left(\frac{\varphi}{1+x}\right)'  dw_1\wedge d\overline{z_d} &\frac{\varphi}{r^2} du\wedge d\overline{z_d} \\
                     \end{array}
                   \right),
\end{eqnarray}
\begin{eqnarray}
 \nonumber  & & (R_{g_F})_{3}\\
 \nonumber  &=& -\frac{(1+x)\varphi}{r^2}\left(\left(\frac{\varphi}{1+x}\right)'-\frac{1}{1+x}\right) \mathbf{e}_1^tdz\wedge d\overline{w_1}+\frac{\varphi}{r^2}(d\bar{w})^t\wedge dz\\
\label{e8.15}   &=&-\left(
                     \begin{array}{ccc}
                        \frac{(1+x)\varphi}{r^2}\left(\frac{\varphi}{1+x}\right)' dz_1\wedge d\overline{w_1} & \cdots & \frac{(1+x)\varphi}{r^2}\left(\frac{\varphi}{1+x}\right)' dz_d\wedge d\overline{w_1} \\
                       \frac{\varphi}{r^2} dz_1\wedge (d\overline{u})^t  & \cdots & \frac{\varphi}{r^2} dz_d\wedge (d\overline{u})^t \\
                     \end{array}
                   \right)
\end{eqnarray}
and
\begin{equation}\label{e8.16}
  (R_{g_F})_{4} =\begin{array}{c}
     \scriptstyle  1 \\
     \scriptstyle  d_0-1
     \end{array}
     \overset{\begin{array}{cr}
           \scriptstyle  1 &\;\;\;\scriptstyle d_0-1
           \end{array}}
     {\left(\begin{array}{c|c}
          (R_{g_F})_{41}    & (R_{g_F})_{42} \\
             \hline
          (R_{g_F})_{43}    &  (R_{g_F})_{44}
           \end{array}
     \right)},
\end{equation}
here
\begin{eqnarray*}
(R_{g_F})_{41}   &=&-\frac{(1+x)\varphi}{r^2}\left(\frac{\varphi}{1+x}\right)'dz\wedge (d\bar{z})^t-\frac{\varphi^2\varphi''}{r^4}dw_1\wedge d\overline{w_1}\\
   & &-\frac{x\varphi}{r^4}\left(\frac{\varphi}{x}\right)'du\wedge (d\overline{u})^t,
\end{eqnarray*}
\begin{equation*}
  (R_{g_F})_{42}=-\frac{x\varphi}{r^4}\left(\frac{\varphi}{x}\right)'du\wedge d\overline{w_1},
\end{equation*}
\begin{equation*}
  (R_{g_F})_{43}=-\frac{x\varphi}{r^4}\left(\frac{\varphi}{x}\right)'dw_1\wedge (d\overline{u})^t
\end{equation*}
and
\begin{eqnarray*}
 (R_{g_F})_{44}  &=& \left\{-\frac{\varphi}{r^2}dz\wedge (d\bar{z})^t-\frac{x\varphi}{r^4}\left(\frac{\varphi}{x}\right)'dw_1\wedge d\overline{w_1}-\frac{\varphi-x}{r^4}du\wedge (d\overline{u})^t\right\}I_{d_0-1}\\
   & & +\frac{\varphi-x}{r^4}(d\overline{u})^t\wedge du.
\end{eqnarray*}

Set
\begin{equation*}
  \phi_{i\bar{j}k\bar{l}}:=\frac{\partial^4\phi}{\partial z_i\partial\bar{z_j}\partial z_k\partial\bar{z_l}},
\end{equation*}
then
\begin{equation}\label{e8.17}
  R_{g_{\phi}}(0)=-\partial\bar{\partial}(\frac{\partial^2\phi}{\partial z^t\partial\bar{z}})(0)=-\left(\sum_{k,l=1}^d\phi_{i\bar{j}k\bar{l}}(0)dz_k\wedge d\bar{z_l}\right)_{1\leq i,j\leq d}.
\end{equation}

Let
\begin{equation*}
  \mathbf{R}_{i\bar{j}}=(R_{i\bar{j}k\bar{l}})_{1\leq k,l\leq d+d_0},\; A_{i\bar{j}}=\mathrm{Tr}\left(T^{-1}\mathbf{R}_{i\bar{j}}T^{-1} \overline{\mathbf{R}_{i\bar{j}}^t}\right),
\end{equation*}

Using
\begin{equation*}
  R_{i\overline{j}k\overline{l}}=R_{k\overline{j}i\overline{l}},\;R_{i\overline{j}k\overline{l}}=R_{i\overline{l}k\overline{j}},\;
   R_{i\overline{j}k\overline{l}}= R_{k\overline{l}i\overline{j}},\; R_{i\overline{j}k\overline{l}}=\overline{ R_{j\overline{i}l\overline{k}}},
\end{equation*}
we get
\begin{equation*}
  A_{i\bar{j}}=A_{j\bar{i}}.
\end{equation*}

By
\begin{equation*}
  T^{-1}=\left(
                \begin{array}{ccc}
                  \frac{1}{1+x}I_d & 0 & 0 \\
                  0 & \frac{r^2}{\varphi} & 0 \\
                  0 & 0 & \frac{r^2}{x}I_{d_0-1} \\
                \end{array}
              \right)
\end{equation*}
at $(0,w)$, we have
\begin{eqnarray*}
\nonumber   |R_{g_F}|^2&=&\sum_{t_1,t_2,t_3,t_4=1}^{d+d_0}\sum_{i,j,k,l=1}^{d+d_0}T^{\overline{j}t_1}T^{\overline{t_2}i}T^{\overline{l}t_3}T^{\overline{t_4}k}R_{t_1\overline{t_2}t_3\overline{t_4}}R_{i\overline{j}k\overline{l}}\\
\nonumber   &=& \sum_{i,j=1}^{d+d_0}(T^{\overline{i}i}T^{\overline{j}j})\sum_{k,l=1}^{d+d_0}(T^{\overline{k}k}T^{\overline{l}l}) |R_{i\overline{j}k\overline{l}}|^2 \\
\nonumber   &=& \sum_{i,j=1}^{d+d_0}(T^{\overline{i}i}T^{\overline{j}j}) A_{i\bar{j}},
\end{eqnarray*}
thus
\begin{eqnarray}
\nonumber |R_{g_F}|^2&=&\sum_{i=1}^d\sum_{j=1}^d\frac{1}{(1+x)^2}A_{i\bar{j}}+2\sum_{i=1}^d\frac{r^2}{(1+x)\varphi}A_{i\overline{d+1}}+2\sum_{i=1}^d\sum_{j=d+2}^{d+d_0}\frac{r^2}{x(1+x)}
A_{i\bar{j}}\\
\label{e8.26}   & &+\frac{r^4}{\varphi^2}A_{d+1\overline{d+1}}+2\sum_{j=d+2}^{d+d_0}\frac{r^4}{x\varphi}A_{d+1\overline{j}}
+\sum_{i=d+2}^{d+d_0}\sum_{j=d+2}^{d+d_0}\frac{r^4}{x^2}A_{i\bar{j}}.
\end{eqnarray}

Let $E_{i\overline{j}}$ be the $d\times d$ matrix whose $(i,j)$ entry is $1$, other entries are $0$; $\mathbf{E}_{i\overline{j}}$ be the $(d+d_0)\times (d+d_0)$ matrix whose $(i,j)$ entry is $1$, other entries are $0$. Set $\mathbf{R}_{i\bar{j}}^{\phi}=(\phi_{i\bar{j}k\bar{l}}(0))_{1\leq,k,l\leq d}$.

From \eqref{e8.12}, \eqref{e8.13}, \eqref{e8.14}, \eqref{e8.15} and \eqref{e8.16}, it follows that

$(i)$ For $1\leq i=j\leq d$,
\begin{equation*}
  \mathbf{R}_{i\bar{i}}=\left(
                          \begin{array}{ccc}
                            -\varphi I_d-(1+x)\mathbf{R}_{i\bar{i}}^{\phi}-\varphi E_{i\overline{i}} & 0 & 0 \\
                            0 & -\frac{(1+x)\varphi}{r^2}\left(\frac{\varphi}{1+x}\right)' & 0 \\
                            0 & 0 & -\frac{\varphi}{r^2}I_{d_0-1} \\
                          \end{array}
                        \right),
\end{equation*}
 so
\begin{eqnarray*}
A_{i\bar{i}}  &=& \frac{1}{(1+x)^2}\mathrm{Tr}\left\{\left(\varphi I_d+(1+x)\mathbf{R}_{i\bar{i}}^{\phi}+\varphi E_{i\overline{i}}\right)
  \left(\varphi I_d+(1+x)\overline{(\mathbf{R}_{i\bar{i}}^{\phi})^t}+\varphi E_{i\overline{i}}\right)\right\}  \\
   && +\frac{r^4}{\varphi^2}\left(\frac{(1+x)\varphi}{r^2}\left(\frac{\varphi}{1+x}\right)'\right)^2
   +(d_0-1) \frac{r^4}{x^2}\frac{\varphi^2}{r^4}\\
  &=&\frac{(d+3)\varphi^2}{(1+x)^2}+\frac{2\varphi}{1+x}\sum_{k=1}^d\phi_{i\bar{i}k\bar{k}}(0)+\sum_{k,l=1}^d|\phi_{i\bar{i}k\bar{l}}(0)|^2
  +\frac{2\varphi}{1+x}\phi_{i\bar{i}i\bar{i}}(0)\\
  &&+  \left((1+x)\left(\frac{\varphi}{1+x}\right)'\right)^2+(d_0-1) \frac{\varphi^2}{x^2}.
\end{eqnarray*}

$(ii)$ For $1\leq i\neq j\leq d$,
\begin{equation*}
  \mathbf{R}_{i\bar{j}}=\left(
                          \begin{array}{lcc}
                           -(1+x)\mathbf{R}_{i\bar{j}}^{\phi}-\varphi E_{j\overline{i}} & 0 & 0 \\
                            0 & 0 & 0 \\
                            0 & 0 & 0 \\
                          \end{array}
                        \right),
\end{equation*}
then
\begin{eqnarray*}
 A_{i\bar{j}}  &=& \frac{1}{(1+x)^2}\mathrm{Tr}\left\{\left((1+x)\mathbf{R}_{i\bar{j}}^{\phi}+\varphi E_{j\overline{i}}\right)
  \left((1+x)\overline{(\mathbf{R}_{i\bar{j}}^{\phi})^t}+\varphi E_{i\overline{j}}\right)\right\}  \\
  &=&\frac{\varphi^2}{(1+x)^2}+\frac{2\varphi}{1+x}\phi_{i\bar{j}j\bar{i}}(0)+\sum_{k,l=1}^d|\phi_{i\bar{j}k\bar{l}}(0)|^2.
\end{eqnarray*}

$(iii)$ For $1\leq i\leq d$, $j=d+1$
\begin{equation*}
   \mathbf{R}_{i\overline{d+1}}=-\frac{(1+x)\varphi}{r^2}\left(\frac{\varphi}{1+x}\right)' \mathbf{E}_{d+1\overline{i}},
\end{equation*}
thus
\begin{equation*}
  A_{i\overline{d+1}}=\frac{r^2}{(1+x)\varphi}\left(\frac{(1+x)\varphi}{r^2}\left(\frac{\varphi}{1+x}\right)'\right)^2=\frac{(1+x)\varphi}{r^2}\left(\left(\frac{\varphi}{1+x}\right)'\right)^2.
\end{equation*}

$(iv)$ For $1\leq i\leq d$, $d+2\leq j\leq d+d_0$,
\begin{equation*}
\mathbf{R}_{i\overline{j}}=-\frac{\varphi}{r^2} \mathbf{E}_{j\overline{i}},
\end{equation*}
we get
\begin{equation*}
  A_{i\overline{j}}=\frac{r^2}{x(1+x)}\left(\frac{\varphi}{r^2}\right)^2=\frac{\varphi^2}{x(1+x)}\frac{1}{r^2}.
\end{equation*}

$(v)$ For $i=j=d+1$, since
\begin{equation*}
\mathbf{R}_{d+1\overline{d+1}}=\left(
                                 \begin{array}{ccc}
                                  - \frac{(1+x)\varphi}{r^2}\left(\frac{\varphi}{1+x}\right)'I_d & 0 & 0 \\
                                   0 & -\frac{\varphi^2\varphi''}{r^4} & 0 \\
                                   0 & 0 & -\frac{x\varphi}{r^4}\left(\frac{\varphi}{x}\right)'I_{d_0-1} \\
                                 \end{array}
                               \right),
\end{equation*}
therefore
\begin{eqnarray*}
  & & A_{d+1\overline{d+1}}\\
  &=& d\frac{1}{(1+x)^2}\left(\frac{(1+x)\varphi}{r^2}\left(\frac{\varphi}{1+x}\right)'\right)^2
  +\frac{r^4}{\varphi^2}\left(\frac{\varphi^2\varphi''}{r^4}\right)^2 +(d_0-1)\frac{r^4}{x^2}\left(\frac{x\varphi}{r^4}\left(\frac{\varphi}{x}\right)'\right)^2\\
   &=&\frac{\varphi^2}{r^4}\left\{d\left(\left(\frac{\varphi}{1+x}\right)'\right)^2+\left(\varphi''\right)^2
   +(d_0-1)\left(\left(\frac{\varphi}{x}\right)'\right)^2\right\}.
\end{eqnarray*}

$(vi)$ For $i=d+1$, $d+2\leq j\leq d+d_0$, using
\begin{equation*}
\mathbf{R}_{d+1\overline{j}}=  -\frac{x\varphi}{r^4}\left(\frac{\varphi}{x}\right)'\mathbf{E}_{j\overline{d+1}},
\end{equation*}
we have
\begin{equation*}
  A_{d+1\overline{j}}=\frac{r^4}{x\varphi}\left(\frac{x\varphi}{r^4}\left(\frac{\varphi}{x}\right)'\right)^2
  =\frac{x\varphi}{r^4}\left(\left(\frac{\varphi}{x}\right)'\right)^2.
\end{equation*}

$(vii)$ For $d+2\leq i=j\leq d+d_0$, by
\begin{equation*}
  \mathbf{R}_{i\overline{i}}=\left(
                               \begin{array}{ccc}
                                 -\frac{\varphi}{r^2}I_d & 0 & 0 \\
                                 0 & -\frac{x\varphi}{r^4}\left(\frac{\varphi}{x}\right)' & 0 \\
                                 0 & 0 & -\frac{\varphi-x}{r^4}I_{d_0-1} \\
                               \end{array}
                             \right)-\frac{\varphi-x}{r^4}\mathbf{E}_{i\overline{i}},
\end{equation*}
we obtain
\begin{eqnarray*}
A_{i\overline{i}}   &=& d\frac{1}{(1+x)^2}\frac{\varphi^2}{r^4}+\frac{r^4}{\varphi^2}\left(\frac{x\varphi}{r^4}\left(\frac{\varphi}{x}\right)'\right)^2
+(d_0+2)\frac{r^4}{x^2}\left(\frac{\varphi-x}{r^4}\right)^2 \\
   &=& \frac{1}{r^4} \left\{d\frac{\varphi^2}{(1+x)^2}+\left(x\left(\frac{\varphi}{x}\right)'\right)^2
+(d_0+2)\left(\frac{\varphi-x}{x}\right)^2\right\}.
\end{eqnarray*}

$(viii)$ For $d+2\leq i\neq j\leq d+d_0$, from
\begin{equation*}
 \mathbf{R}_{i\overline{j}}=-\frac{\varphi-x}{r^4}\mathbf{E}_{j\overline{i}},
\end{equation*}
then
\begin{equation*}
 A_{i\overline{j}}=\frac{r^4}{x^2}\frac{(\varphi-x)^2}{r^8}=\frac{1}{r^4}\left(\frac{\varphi-x}{x}\right)^2.
\end{equation*}

Combining the above $(i)-(viii)$, from \eqref{e8.26}, we have
\begin{eqnarray*}
|R_{g_F}|^2 &=&\frac{4\varphi}{(1+x)^3}\sum_{i=1}^d\sum_{k=1}^d\phi_{i\bar{i}k\bar{k}}(0)+\frac{1}{(1+x)^2}\sum_{i,j,k,l=1}^d|\phi_{i\bar{j}k\bar{l}}(0)|^2
 +2d(d+1)\frac{\varphi^2}{(1+x)^4}\\
&& +4d\left(\left(\frac{\varphi}{1+x}\right)'\right)^2+\left(\varphi''\right)^2+4d(d_0-1) \left(\frac{\varphi}{x(1+x)}\right)^2\\
&& +4(d_0-1)\left(\left(\frac{\varphi}{x}\right)'\right)^2+2d_0(d_0-1)\left(\frac{\varphi-x}{x^2}\right)^2.
\end{eqnarray*}

Since
\begin{equation*}
  k_{g_{\phi}}(0)=-\sum_{i,k=1}^d\phi_{i\overline{i}k\overline{k}}(0),\; | R_{g_{\phi}}|^2(0)=\sum_{i,j,k,l=1}^d|\phi_{i\bar{j}k\bar{l}}(0)|^2,
\end{equation*}
we get
\begin{eqnarray}
\nonumber   & &  |R_{g_F}|^2\\
\nonumber   &=&\frac{1}{(1+x)^2}|R_{g_{\phi}}|^2(0)-\frac{4\varphi}{(1+x)^3}k_{g_{\phi}}(0)+2d(d+1)\frac{\varphi^2}{(1+x)^4}+4d\left(\left(\frac{\varphi}{1+x}\right)'\right)^2\\
\nonumber   & &+\left(\varphi''\right)^2+(d_0-1)\left\{4d \left(\frac{\varphi}{x(1+x)}\right)^2+4\left(\left(\frac{\varphi}{x}\right)'\right)^2+2d_0\left(\frac{\varphi-x}{x^2}\right)^2\right\},
\end{eqnarray}
which completes the proof of \eqref{e8.1}.
\end{proof}

Applying Lemma \ref{Le:5.2} and Lemma \ref{Le:5.3}, we obtain explicit expressions of the coefficients $\mathbf{a}_j\;(j=1,2)$ of the Bergman function expansion for $ (M,g_{F})$.

\begin{Theorem}\label{Th:5.2}{Assume that $\phi$
is a globally defined  K\"{a}hler potential on a domain $\Omega\subset \mathbb{C}^d$.
Let $g_{\phi}$ be a K\"{a}hler  metric  on the domain $\Omega$ associated with the
K\"{a}hler form $\omega_{\phi}=\frac{\sqrt{-1}}{2\pi}\partial\overline{\partial}\phi$, and
$g_F$ be a K\"{a}hler  metric  on the domain $ M$ associated with the K\"{a}hler form
$\omega_F=\frac{\sqrt{-1}}{2\pi}\partial\overline{\partial}\Phi_F$, where
$$\Phi_F(z,w)=\phi(z)+F(\phi(z)+\ln\|w\|^2)$$
is a K\"{a}hler potential on a Hartogs domain
$$ M=\left\{(z,w)\in \Omega\times\mathbb{C}^{d_0}: \|w\|^2<e^{-\phi(z)}\right\}$$

Set
$$t=\phi(z)+\ln\|w\|^2,\;\;x=F'(t),\;\varphi(x)=F''(t)\;,$$
\begin{equation*}
  \sigma(x)=\frac{\left((1+x)^dx^{d_0-1}\varphi(x)\right)'}{(1+x)^dx^{d_0-1}}
\end{equation*}
and
\begin{equation*}
   \chi(x)=\frac{d_0(d_0-1)}{x}-\frac{\left((1+x)^dx^{d_0-1}\varphi(x)\right)''}{(1+x)^dx^{d_0-1}}.
\end{equation*}

Let $k_{g_{\phi}}$, $\Delta_{g_{\phi}}$, $\mathrm{Ric}_{g_{\phi}}$ and $R_{g_{\phi}}$ be the scalar curvature, the Laplace, the Ricci curvature and the curvature tensor on the domain $\Omega$ with respect to the metric $g_{\phi}$, respectively. Put
\begin{equation*}
 a_1=\frac{1}{2}k_{g_{\phi}},\;a_2=\frac{1}{3}\triangle_{g_{\phi}} k_{g_{\phi}}+\frac{1}{24}|R_{g_{\phi}}|^2-\frac{1}{6}|\mathrm{Ric}_{g_{\phi}}|^2+\frac{1}{8}k_{g_{\phi}}^2
\end{equation*}
and
\begin{equation*}
  \mathbf{a}_1=\frac{1}{2}k_{g_F},\;\mathbf{a}_2=\frac{1}{3}\triangle_{g_F} k_{g_F}+\frac{1}{24}|R_{g_F}|^2-\frac{1}{6}|\mathrm{Ric}_{g_F}|^2+\frac{1}{8}k_{g_F}^2.
\end{equation*}

 Then
\begin{equation}\label{e5.48}
   \mathbf{a}_1=\frac{1}{1+x}a_1+\frac{d_0(d_0-1)}{2x}-\frac{\left((1+x)^dx^{d_0-1}\varphi\right)''}{2(1+x)^dx^{d_0-1}}
\end{equation}
and
\begin{eqnarray}
\nonumber  \mathbf{a}_2  &=& \frac{1}{(1+x)^2}a_2+\left\{\frac{1}{2(1+x)}\chi+\frac{\varphi}{(1+x)^3}\right\}a_1 \\
\nonumber   & & +\frac{1}{24}\left\{8(\varphi\chi')'+3\chi^2-4(\sigma')^2-\frac{4d}{(1+x)^2}\sigma^2+(\varphi'')^2\right.\\
\nonumber  & &\left.
+4d\left(\left(\frac{\varphi}{1+x}\right)'\right)^2+8\frac{(d+d_0-1)x+(d_0-1)}{x(1+x)}\varphi\chi'+\frac{2d(d+1)}{(1+x)^4}\varphi^2\right\}\\
\label{e5.49}   & &+\frac{d_0-1}{6}\left\{\frac{d\varphi^2}{x^2(1+x)^2}+\left(\left(\frac{\varphi}{x}\right)'\right)^2+\frac{d_0}{2}\frac{(\varphi-x)^2}{x^4}
-\frac{(\sigma-d_0)^2}{x^2}\right\}.
\end{eqnarray}
}\end{Theorem}

Theorem \ref{Th:5.2} implies the following results.

\begin{Theorem}\label{Th:3.2}{Assume that
\begin{equation*}
  \varphi(x)=x(1+x),
\end{equation*}
in Theorem \ref{Th:5.2}, then

$(\mathrm{I})$
\begin{equation}\label{e3.48}
   \mathbf{a}_1=\left(a_1+\frac{d(d+1)}{2}\right)\frac{1}{1+x}-\frac{n(n+1)}{2}
\end{equation}
and
\begin{eqnarray}
\nonumber  \mathbf{a}_2 &=& \left\{a_2+\frac{(d-1)(d+2)}{2}\left(a_1+\frac{d(d+1)}{2}\right)-\frac{(d-1)d(d+1)(3d+2)}{24}\right\} \\
\nonumber   & &\times\frac{1}{(1+x)^2}-\frac{(n-1)(n+2)}{2}\left(a_1+\frac{d(d+1)}{2}\right)\frac{1}{1+x}\\
\label{e3.49}   & &  +\frac{(n-1)n(n+1)(3n+2)}{24},
\end{eqnarray}
where $n=d+d_0$.

 $(\mathrm{II})$ $\mathbf{a}_1$ is a constant $\Longleftrightarrow$
$a_1=-\frac{d(d+1)}{2}$  $\Longleftrightarrow$ $\mathbf{a}_1=-\frac{n(n+1)}{2}$.

$(\mathrm{III})$ $\mathbf{a}_2$ is a constant $\Longleftrightarrow$
$\mathbf{a}_2=\frac{(n-1)n(n+1)(3n+2)}{24}$ $\Longleftrightarrow$
$a_1=-\frac{d(d+1)}{2}$ and $a_2=\frac{(d-1)d(d+1)(3d+2)}{24}$. }\end{Theorem}

\begin{Theorem}\label{Th:3.3}{Assume that $d=1$ and
\begin{equation*}
 \varphi(x)=Ax^2+x
\end{equation*}
in Theorem \ref{Th:5.2}, then

$(\mathrm{I})$
\begin{equation*}
   \mathbf{a}_1=\frac{a_1 - d_0 + A(d_0+1)}{1+x}-\frac{n(n+1)}{2}A.
\end{equation*}
If $a_1=d_0-(d_0+1)A$, then
\begin{equation*}
 \mathbf{a}_2=\frac{a_2}{(1+x)^2} +\frac{(n-1)n(n+1)(3n+2)}{24}A^2.
\end{equation*}
Here $n=1+d_0$.

 $(\mathrm{II})$ $\mathbf{a}_1$ is a constant $\Longleftrightarrow$ $a_1=d_0-(d_0+1)A$  $\Longleftrightarrow$ $\mathbf{a}_1=-\frac{n(n+1)}{2}A$.

$(\mathrm{III})$ Both $\mathbf{a}_1$ and $\mathbf{a}_2$ are constants  $\Longleftrightarrow$
$a_1=d_0-(d_0+1)A$ and $a_2=0$ $\Longleftrightarrow$ $ \mathbf{a}_1=-\frac{n(n+1)}{2}A$ and
$\mathbf{a}_2=\frac{(n-1)n(n+1)(3n+2)}{24}A^2$. }\end{Theorem}

\setcounter{equation}{0}
\section{The K\"{a}hler metrics with constant  coefficients $\mathbf{a}_j\;(j=1,2)$  }

In this section, we first give an explicit expression of the function $\varphi(x)$ for the coefficients $\mathbf{a}_j\;(j=1,2)$ to be constants, then we get an explicit expression of $F(t)$ by solving the differential equation. Finally we complete the proof of Theorem \ref{apth:3.3}.

\begin{Theorem}\label{Th:5.2.1}{
Assume that $\phi$ is a globally defined  K\"{a}hler potential on a domain $\Omega\subset \mathbb{C}^d$. Let $g_{\phi}$ be a K\"{a}hler  metric  on the domain $\Omega$ associated with the K\"{a}hler form $\omega=\frac{\sqrt{-1}}{2\pi}\partial\overline{\partial}\phi$.

Let $g_{F}$ be a K\"{a}hler  metric on the Hartogs domain
\begin{equation*}
  M=\left\{(z,w)\in \Omega\times\mathbb{C}^{d_0}: \|w\|^2<e^{-\phi(z)}\right\}
\end{equation*}
 associated with the K\"{a}hler form
 $$\omega_F=\frac{\sqrt{-1}}{2\pi}\partial\overline{\partial}\Phi_{F},$$
where $\|\cdot\|$ is the standard Hermite norm in $\mathbb{C}^{d_0}$ and
$$\Phi_{F}(z,w)=\phi(z)+F(t),\; t=\phi(z)+\ln\|w\|^2,$$
where $F(t)$ satisfies the following conditions:

$(\mathrm{i})$
\begin{equation}\label{ap1.1}
 \lim_{t\rightarrow -\infty}F(t)=0.
\end{equation}

$(\mathrm{ii})$ For the case of $d_0>1$,
\begin{equation}\label{ap1.2}
0< \lim_{t\rightarrow -\infty}\frac{F'(t)}{e^t}<+\infty, \;0< \lim_{t\rightarrow -\infty}\frac{F''(t)}{e^t}<+\infty,\;F'(t)>0,\;F''(t)>0,\;t\in (-\infty,0).
\end{equation}

$(\mathrm{iii})$ For the case of $d_0=1$,
\begin{equation}\label{ap1.3}
0<\lim_{t\rightarrow -\infty}(1+F'(t))<+\infty, \;0< \lim_{t\rightarrow -\infty}\frac{F''(t)}{e^t}<+\infty,\;1+F'(t)>0,\;F''(t)>0,\;t\in (-\infty,0).
\end{equation}

Let $k_{g_{\phi}}$, $\Delta_{g_{\phi}}$, $\mathrm{Ric}_{g_{\phi}}$ and $R_{g_{\phi}}$ be the scalar curvature, the Laplace, the Ricci curvature and the curvature tensor on the domain $\Omega$ with respect to the metric $g_{\phi}$, respectively. Put
\begin{equation*}
 a_1=\frac{1}{2}k_{g_{\phi}},\;a_2=\frac{1}{3}\triangle_{g_{\phi}} k_{g_{\phi}}+\frac{1}{24}|R_{g_{\phi}}|^2-\frac{1}{6}|\mathrm{Ric}_{g_{\phi}}|^2+\frac{1}{8}k_{g_{\phi}}^2
\end{equation*}
and
\begin{equation*}
  \mathbf{a}_1=\frac{1}{2}k_{g_F},\;\mathbf{a}_2=\frac{1}{3}\triangle_{g_F} k_{g_F}+\frac{1}{24}|R_{g_F}|^2-\frac{1}{6}|\mathrm{Ric}_{g_F}|^2+\frac{1}{8}k_{g_F}^2.
\end{equation*}

Set
$$x=F'(t),\;\varphi(x)=F''(t),$$
then  we have the following results.

$(\mathrm{I})$ For $d_0=1$, let
\begin{equation*}
 a_1=-\frac{1}{2}d(d+1)B,\;\mathbf{a}_1=-\frac{1}{2}(d+1)(d+2)A.
\end{equation*}
$\mathbf{a}_1$ is a constant if and only if $a_1$ is a constant and
\begin{equation}\label{e5.61}
  \varphi(x)=A(1+x)^2-B(1+x)+\frac{C_1}{(1+x)^{d-1}}+\frac{C_2}{(1+x)^d},
\end{equation}
where $C_1$ and $C_2$ are constants.

$(\mathrm{II})$ For $d_0>1$ and $d=1$, $\mathbf{a}_1$ is a constant if and only if $a_1$ is
a constant and
\begin{equation}\label{e5.61.1}
  \varphi(x)=\left(\frac{d_0-a_1}{d_0+1}+\frac{d_0}{2}C\right)x^2+(C+1)x-C+\frac{C}{1+x},
\end{equation}
here $C$ is a constant.

$(\mathrm{III})$ For $d_0=1$, let
\begin{equation*}
  a_1=-\frac{1}{2}d(d+1)B,\;\mathbf{a}_1=-\frac{1}{2}(d+1)(d+2)A.
\end{equation*}
 If $\mathbf{a}_1$ and $\mathbf{a}_2$ are constants, then $a_1$ and $a_2$ are constants,
\begin{equation}\label{e5.63}
a_2=\frac{1}{24}(d-1)d(d+1)(3d+2)B^2,\;\mathbf{a}_2=\frac{1}{24}d(d+1)(d+2)(3d+5)A^2
\end{equation}
and
\begin{equation}\label{e5.62}
  \varphi(x)=A(1+x)^2-B(1+x)
\end{equation}
for $d>1$,
\begin{equation}\label{e5.64}
  \varphi(x)=A(1+x)^2-B(1+x)+C
\end{equation}
for $d=1$, here $A,B, C$ are  constants.

$(\mathrm{IV})$ For $d_0>1$ and $d=1$, let $A=\frac{d_0-a_1}{d_0+1}$. If $\mathbf{a}_1$ and $\mathbf{a}_2$ are
constants, then $a_1$ is a constant,
\begin{equation*}
 a_2=0,\; \mathbf{a}_1=-A\frac{(d_0+1)(d_0+2)}{2},\;\mathbf{a}_2=A^2\frac{d_0(d_0+1)(d_0+2)(3d_0+5)}{24}
\end{equation*}
and
\begin{equation*}
  \varphi(x)=Ax^2+x.
\end{equation*}

$(\mathrm{V})$ For $d_0>1$ and $d>1$, let $n=d+d_0$. If $\mathbf{a}_1$ and $\mathbf{a}_2$ are constants, then
\begin{equation*}
  \varphi(x)=x(x+1),
\end{equation*}
\begin{equation*}
  a_1=-\frac{d(d+1)}{2},\;a_2=\frac{(d-1)d(d+1)(3d+2)}{24}
\end{equation*}
and
\begin{equation*}
  \mathbf{a}_1 =-\frac{n(n+1)}{2},\;\mathbf{a}_2=\frac{(n-1)n(n+1)(3n+2)}{24}.
\end{equation*}
}\end{Theorem}

\begin{proof}[Proof]

$(\mathrm{I})$ Using \eqref{e5.48}, we get \eqref{e5.61}.

$(\mathrm{II})$ For $d_0>1$ and $d=1$, if $\mathbf{a}_1$ is a constant, from \eqref{e5.48}, we obtain $a_1$ is a constant, and $\varphi(x)$ can be written as
\begin{equation*}
  \varphi(x)=\sum_{j=0}^2A_j(1+x)^j+\sum_{j=1}^{d_0-1}\frac{B_j}{x^j}+\frac{C}{1+x}.
\end{equation*}

Since
$$\lim_{t\rightarrow -\infty}F'(t)=\lim_{t\rightarrow -\infty}F''(t)=0,$$
  we have $\varphi(0)=0$, thus
\begin{equation*}
B_j=0 \;(1\leq j\leq d_0-1),\;\sum_{j=0}^2A_j+C=0.
\end{equation*}
By \eqref{e5.48}, we get
\begin{eqnarray*}
   & &\frac{\left((1+x)^dx^{d_0-1}\varphi(x)\right)''}{(1+x)^dx^{d_0-1}} -\frac{d_0(d_0-1)}{x}-\frac{2a_1}{1+x}+2\mathbf{a}_1 \\
   &=&(1+d_0)(2+d_0)A_2+2\mathbf{a}_1+d_0(d_0-1)\frac{2A_2+A_1-C-1}{x}\\
    &&+\frac{2A_1-2(d_0-1)A_0-2a_1+(d_0-1)(d_0-2)C}{1+x}\\
    &\equiv &0,
\end{eqnarray*}
thus
\begin{equation*}
  \mathbf{a}_1=-\frac{1}{2}(1+d_0)(2+d_0)A_2,\;2A_2+A_1-C-1=0
\end{equation*}
and
\begin{equation*}
 2A_1-2(d_0-1)A_0-2a_1+(d_0-1)(d_0-2)C=0.
\end{equation*}

In view of
$$A_0+A_1+A_2+C=0,$$
we obtain
$$A_2=\frac{1}{2}d_0C+\frac{d_0-a_1}{d_0+1},$$
then
\begin{eqnarray*}
  \varphi(x) &=&\sum_{j=0}^2A_j(1+x)^j+\frac{C}{1+x}    \\
   &=&A_2x^2+(A_1+2A_2)x+(A_0+A_1+A_2)+\frac{C}{1+x}  \\
   &=& \left(\frac{d_0-a_1}{d_0+1}+\frac{d_0}{2}C\right)x^2+(C+1)x-C+\frac{C}{1+x}.
\end{eqnarray*}

$(\mathrm{III})$ By  \eqref{e5.48} and \eqref{e5.49}, if $\mathbf{a}_1$ and
$\mathbf{a}_2$ are constants, then $a_1$ and $a_2$ are constants, thus
$|R_{g_{\phi}}|^2-4|\mathrm{Ric}_{g_{\phi}}|^2$ and
$|R_{g_F}|^2-4|\mathrm{Ric}_{g_F}|^2$ also are constants.

Substituting \eqref{e5.61} into \eqref{e8.1} and \eqref{e5.55}, we have
\begin{eqnarray*}
   & & |R_{g_F}|^2-4|\mathrm{Ric}_{g_F}|^2 \\
   &=& \frac{1}{(1+x)^2}\left(|R_{g_{\phi}}|^2-4|\mathrm{Ric}_{g_{\phi}}|^2\right)+\left(\frac{8\sigma}{(1+x)^2}-\frac{4\varphi}{(1+x)^3}\right)k_{g_{\phi}} \\
   & &+2d(d+1)\frac{\varphi^2}{(1+x)^4}+4d\left(\left(\frac{\varphi}{1+x}\right)'\right)^2 + \left(\varphi''\right)^2-4\left(\sigma'\right)^2-\frac{4d}{(1+x)^2}\sigma^2\\
   &=&\frac{A_{2d+4}}{(1+x)^{2d+4}}+\frac{A_{2d+3}}{(1+x)^{2d+3}}+\frac{A_{2d+2}}{(1+x)^{2d+2}}+\frac{A_{2}}{(1+x)^{2}}+A_0,
\end{eqnarray*}
where
\begin{equation*}
  A_{2d+4}=d(d+1)(d+2)(d+3)C_2^2,\; A_{2d+3}=2d(d+1)^2(d+2)C_1C_2,
\end{equation*}
\begin{equation*}
  A_{2d+2}=(d-1)d(d+1)(d+2)C_1^2,\; A_2=|R_{g_{\phi}}|^2-4|\mathrm{Ric}_{g_{\phi}}|^2+\frac{4d+2}{d(d+1)}k_{g_{\phi}}^2
\end{equation*}
and
\begin{equation*}
  A_0=-2(d+1)(d+2)(2d+3)A^2.
\end{equation*}

From $A_{2d+4}=0$, we have $C_2=0$, so $A_{2d+3}=0$.

For $d>1$, by $A_{2d+2}=0$,  it follows that $C_1=0$.

For $d=1$, for any $C_1$, $A_{2d+2}=0$.

Substituting $\varphi(x)=A(1+x)^2-B(1+x)$ and $\varphi(x)=A(1+x)^2-B(1+x)+C$ into \eqref{e5.49}, we obtain
\begin{equation*}
  \mathbf{a}_2=\frac{a_2-\frac{1}{24}(d-1)d(d+1)(3d+2)B^2}{(1+x)^2}+\frac{d(d+1)(d+2)(3d+5)A^2}{24}
\end{equation*}
for $d>1$, and
\begin{equation*}
  \mathbf{a}_2=2A^2 + \frac{a_2}{(x + 1)^2}
\end{equation*}
for $d=1$, respectively. This gives \eqref{e5.63}.

$(\mathrm{IV})$ Using \eqref{e5.61.1} and \eqref{e5.49}, we get
\begin{equation*}
  \mathbf{a}_2=\frac{C^2}{(1+x)^6}+\sum_{j=0}^5\frac{F_j}{(1+x)^j},
\end{equation*}
thus
\begin{equation*}
  C=0,\;\varphi(x)=\frac{d_0-a_1}{d_0+1}x^2+x.
\end{equation*}
Then put them into \eqref{e5.48} and \eqref{e5.49}, we have
\begin{equation*}
  \mathbf{a}_1=\frac{(a_1-d_0)(d_0+2)}{2}
\end{equation*}
and
\begin{equation*}
  \mathbf{a}_2=\frac{a_2}{(1+x)^2}+\frac{d_0(d_0+2)(3d_0+5)(a_1-d_0)^2}{24(d_0+1)}.
\end{equation*}

Let
\begin{equation*}
  A=\frac{d_0-a_1}{d_0+1},
\end{equation*}
then
\begin{equation*}
 a_2=0,\; \mathbf{a}_1=-A\frac{(d_0+1)(d_0+2)}{2},\;\mathbf{a}_2=A^2\frac{d_0(d_0+1)(d_0+2)(3d_0+5)}{24}.
\end{equation*}

$(\mathrm{V})$ For $d>1$ and $d_0>1$, according to \eqref{e5.48}, we obtain
\begin{eqnarray*}
   & &  ((1+x)^dx^{d_0-1}\varphi)'' \\
   &=& 2a_1(1+x)^{d-1}x^{d_0-1}+d_0(d_0-1)(1+x)^dx^{d_0-2}-2\mathbf{a}_1(1+x)^dx^{d_0-1},
\end{eqnarray*}
using integration by parts, we obtain
\begin{eqnarray*}
   & & (1+x)^dx^{d_0-1}\varphi(x) \\
   &=& \sum_{j=0}^{d_0-1}c_{j1}x^j(1+x)^{d+d_0-j}+\sum_{j=0}^{d_0-1}c_{j2}x^j(1+x)^{d+d_0+1-j}+(-1)^{d_0-1}(C_1x+C_2),
\end{eqnarray*}
thus
\begin{equation*}
  \varphi(x)=\sum_{j=0}^{d_0-1}d_{j1}\frac{(1+x)^{j+1}}{x^j}+\sum_{j=0}^{d_0-1}d_{j2}\frac{(1+x)^{j+2}}{x^j}+(-1)^{d_0-1}\frac{C_1x+C_2}{(1+x)^dx^{d_0-1}},
\end{equation*}
which can be written as
\begin{equation*}
  \varphi(x)=\sum_{j=0}^2A_j(x+1)^j+\sum_{j=1}^{d_0-1}\frac{B_j(a_1,\mathbf{a}_1,d,d_0)}{x^j}+(-1)^{d_0-1}\frac{C_1x+C_2}{x^{d_0-1}(1+x)^d}.
\end{equation*}

Owing to $\varphi(0)=0$ and
\begin{eqnarray*}
   & & (-1)^{d_0-1} \frac{C_1x+C_2}{x^{d_0-1}(1+x)^d} \\
   &=& \sum_{j=1}^{d_0-1}\frac{D_j}{x^j}+\frac{C_1-C_2}{(1+x)^d}+\frac{(d_0-1)C_2-(d_0-2)C_1}{(1+x)^{d-1}}+\sum_{j=1}^{d-2}\frac{E_j}{(1+x)^j},
\end{eqnarray*}
we have
\begin{equation*}
  \varphi(x)=\sum_{j=0}^2A_j(x+1)^j+\frac{C_1-C_2}{(1+x)^d}+\frac{(d_0-1)C_2-(d_0-2)C_1}{(1+x)^{d-1}}+\sum_{j=1}^{d-2}\frac{E_j}{(1+x)^{j}}.
\end{equation*}

From \eqref{e5.55} and \eqref{e8.1}, we have

\begin{equation*}
  |R_{g_F}|^2-4|\mathrm{Ric}_{g_F}|^2=\frac{d(d+1)(d+2)(d+3)(C_1-C_2)^2}{(1+x)^{2d+4}}+\sum_{j=1}^{2d+3}\frac{F_j}{(1+x)^j}+\sum_{j=0}^4\frac{G_j}{x^j},
\end{equation*}
which implies that
\begin{equation*}
  C_1=C_2,
\end{equation*}
so
\begin{equation*}
  |R_{g_F}|^2-4|\mathrm{Ric}_{g_F}|^2=\frac{(d-1)d(d+1)(d+2)C_1^2}{(1+x)^{2d+2}}+\sum_{j=1}^{2d+1}\frac{F_j}{(1+x)^j}+\sum_{j=0}^4\frac{G_j}{x^j},
\end{equation*}
therefore
\begin{equation*}
  C_1=C_2=0,
\end{equation*}
namely
\begin{equation}\label{e5.75}
  \varphi(x)=\sum_{j=0}^2A_j(1+x)^j.
\end{equation}

By \eqref{e5.75} and \eqref{e5.48}, we obtain
\begin{equation*}
  \mathbf{a}_1=-\frac{d(d-1)}{2(1+x)^2}A_0+\frac{H_1}{x}+\frac{H_2}{x^2}+\frac{H_3}{1+x}+H_0,
\end{equation*}
since $\mathbf{a}_1$ is a constant, then
\begin{equation*}
  A_0=0,
\end{equation*}
which combines with $\varphi(0)=0$, we get
$A_1+A_2=0,$
thus
$\varphi(x)=A_2x(1+x).$

Substituting
$$\varphi(x)=A_2x(1+x)$$
into \eqref{e5.48}, we have
\begin{equation*}
   \mathbf{a}_1=-\frac{(d+d_0)(d+d_0+1)}{2}A_2+ \frac{d(d+ 1)A_2 + 2a_1}{2(x + 1)}-\frac{(d_0 - 1)d_0(A_2-1)}{2x},
\end{equation*}
notice that $d_0>1$, thus
\begin{equation*}
  A_2=1,\;a_1=-\frac{d(d+ 1)}{2},\; \mathbf{a}_1 =-\frac{(d+d_0)(d+d_0+1)}{2}.
\end{equation*}

Finally, from \eqref{e3.49}, we obtain
\begin{equation*}
  a_2=\frac{(d-1)d(d+1)(3d+2)}{24},\;\mathbf{a}_2=\frac{(n-1)n(n+1)(3n+2)}{24},
\end{equation*}
where $n=d+d_0$.
\end{proof}

\begin{Theorem}\label{Th:5.5}{
Under assumptions of Theorem \ref{Th:5.2.1}, we have the following results.

$(\mathrm{i})$ For $d=1$,  if $\mathbf{a}_1$ and
$\mathbf{a}_2$ are constants, then
\begin{equation}\label{ap3.1}
  \varphi(x)= Ax^2+x
\end{equation}
and
\begin{equation}\label{ap3.2}
  F(t)=\left\{\begin{array}{ll}
              -\frac{1}{A}\ln(1-ce^{t}),   &cA>0,c\leq 1,  \\
              ce^{t},   & c>0,A=0.
              \end{array}
  \right.
\end{equation}

$(\mathrm{ii})$ For $d>1$, if $\mathbf{a}_1$ and $\mathbf{a}_2$ are
constants, then
\begin{equation}\label{ap3.3}
\varphi(x)=x(x+1)
\end{equation}
and
\begin{equation}\label{ap3.4}
 F(t)=-\ln\left(1-ce^{t}\right),\;0<c\leq 1.
\end{equation}
}\end{Theorem}

\begin{proof}[Proof]

$(\mathrm{I})$
Let $$ x=F'(t),\;\varphi(x)=F''(t),$$
according to Theorem \ref{Th:5.2.1}, we can assume that
 \begin{equation*}
   \varphi(x)=Ax^2+Dx+E.
 \end{equation*}

Using $\frac{dx}{dt}=\varphi(x)$, $x=F'(t)$ and $F''(t)>0$, $\varphi(x)$ and $F(t)$ can be expressed as follows:
\begin{equation}\label{ap3.6}
  \varphi(x)=\left\{\begin{array}{ll}
                    A(x-\lambda)^2,   & A>0, \\
                    A((x-\lambda)^2+\mu^2),   & A>0,\mu>0, \\
                    A((x-\lambda)^2-\mu^2),   & \mu A>0, \\
                    2\lambda,   & \lambda>0, \\
                    \lambda(x-\mu),   &  \lambda\neq 0,
                    \end{array}
  \right.
\end{equation}
\begin{equation}\label{3.6}
  F(t)=\left\{\begin{array}{ll}
              -\frac{1}{A}\ln|t+c|+\lambda t+c_1,   &A>0,  \\
              -\frac{1}{A}\ln|\cos(\mu At+c)|+\lambda t+c_1,   &  A>0,\mu>0, \\
              -\frac{1}{A}\ln\left|1-ce^{2\mu At}\right|+(\lambda+\mu)t+c_1,   & \mu A>0,cA>0, \\
              \lambda t^2+ct+c_1,   &\lambda>0,  \\
              ce^{\lambda t}+\mu t+c_1,   & \lambda\neq 0,c>0.
              \end{array}
  \right.
\end{equation}

$(\mathrm{II})$ For the case of $d=d_0=1$, by $F(-\infty)=0$ and \eqref{3.6}, $\varphi(x)$ and $F(t)$ can be expressed as
\begin{equation*}
  \varphi(x)=\left\{\begin{array}{ll}
                     Ax(x+2\mu),  & \mu A>0, \\
                      \lambda x, &  \lambda>0,
                    \end{array}
  \right.
\end{equation*}
\begin{equation*}
  F(t)=\left\{\begin{array}{ll}
              -\frac{1}{A}\ln\left|1-ce^{2\mu At}\right|,   &\mu A>0,cA>0,c\leq 1,  \\
              ce^{\lambda t},   & c>0,\lambda>0,
              \end{array}
  \right.
\end{equation*}
respectively. In view of
\begin{equation*}
 0< \lim_{t\rightarrow -\infty}\frac{F''(t)}{e^t}<+\infty,
\end{equation*}
we get $2\mu A=1$ for $F(t)=-\frac{1}{A}\ln\left|1-ce^{2\mu At}\right|$, and $\lambda=1$ for $F(t)=ce^{\lambda t}$. Thus $\varphi(x)$ and $F(t)$ can be written as \eqref{ap3.1} and \eqref{ap3.2}, respectively.

$(\mathrm{III})$ For the case of $d=1$ and $d_0>1$, applying $\varphi(x)=Ax^2+x$, $F(-\infty)=0$ and \eqref{3.6}, we obtain \eqref{ap3.1} and \eqref{ap3.2}.

$(\mathrm{IV})$ For the case of $d>1$ and $d_0=1$, from $\varphi(x)=A(1+x)^2-B(1+x)$ and $F(-\infty)=0$, by \eqref{3.6}, we have
\begin{equation*}
  \varphi(x)=Ax(x+2\mu), \mu A>0
\end{equation*}
and
\begin{equation*}
  F(t)=-\frac{1}{A}\ln\left|1-ce^{2\mu At}\right|,  \mu A>0,cA>0,c\leq 1.
\end{equation*}
Since
\begin{equation*}
 0< \lim_{t\rightarrow -\infty}\frac{F''(t)}{e^t}<+\infty,
\end{equation*}
then $2\mu A=1$. Thus $\varphi(x)=Ax^2+x$, using $\varphi(-1)=0$, we get \eqref{ap3.3} and \eqref{ap3.4}.

$(\mathrm{V})$ For the case of $d>1$ and $d_0>1$, from $\varphi(x)=x^2+x$ and $F(-\infty)=0$, by \eqref{3.6}, we have \eqref{ap3.4}.
\end{proof}

\begin{proof}[Proof of Theorem \ref{apth:3.3}]
According to Theorem \ref{Th:5.5}, Theorem \ref{Th:3.2} and  Theorem \ref{Th:3.3}, we obtain that both $\mathbf{a}_1$ and $\mathbf{a}_2$ are constants on $(M,g_F)$ if and only if

 $(\mathrm{i})$
\begin{equation}\label{ap3.16}
 F(t)=\left\{\begin{array}{ll}
              -\frac{1}{A}\ln(1-ce^t),   &cA>0,c\leq 1,  \\
              ce^t,   & A=0,c>0
              \end{array}
  \right.
\end{equation}
and $a_2=0$ for $d=1$, where $A=\frac{d_0-a_1}{d_0+1}$.

 $(\mathrm{ii})$
\begin{equation}\label{ap3.17}
 F(t)=-\ln\left(1-ce^t\right),\;0<c\leq 1,
\end{equation}
$a_1=-\frac{d(d+1)}{2}$ and $a_2=\frac{(d-1)d(d+1)(3d+2)}{24}$ for $d>1$.

Notice that we do not consider the completeness of the metrics $g_{\phi}$ and $g_F$ in the formulas \eqref{ap3.16} and \eqref{ap3.17}.

Let
\begin{equation*}
  f(u)=\frac{1}{2}F(2u),\;\tau=f'(u),\;\varphi(\tau)=f''(u),
\end{equation*}
then
\begin{equation*}
  \varphi(\tau)=2\tau(1+A\tau),\;\tau\in I=\left[0,\frac{1}{A}\frac{c}{1-c}\right)
\end{equation*}
for
\begin{equation*}
  F(t)= -\frac{1}{A}\ln(1-ce^t),  \;cA>0,\;c\leq 1,\;t\in[-\infty,0),
\end{equation*}
and
\begin{equation*}
  \varphi(\tau)=2\tau,\;\tau\in I=\left[0,c\right)
\end{equation*}
for
\begin{equation*}
  F(t)= ce^t,\;c>0,\;t\in[-\infty,0).
\end{equation*}

By Proposition 2.3 of \cite{Hwang-Singer}, $g_F$ is a complete K\"{a}hler  metric on the domain $M$ for $d_0=1$ if and only if
\begin{equation*}
 F(t)=\left\{\begin{array}{ll}
              -\frac{1}{A}\ln(1-e^t),   &d=1,A>0,  \\
             -\ln\left(1-e^t\right),   & d>1.
              \end{array}
  \right.
\end{equation*}
Using mathematical induction, we obtain that $g_F$ is a complete K\"{a}hler  metric on the domain $M$ if and only if
\begin{equation*}
 F(t)=\left\{\begin{array}{ll}
              -\frac{1}{A}\ln(1-e^t),   &d=1,A>0,  \\
             -\ln\left(1-e^t\right),   & d>1.
              \end{array}
  \right.
\end{equation*}
So we complete the proof of Theorem \ref{apth:3.3}.
\end{proof}

\setcounter{equation}{0}
\section{Proof of Theorem \ref{Th:1.1}}

 In order to prove Theorem \ref{Th:1.1}, we first give the K\"{a}hler forms of  $\mathcal{G}$-invariant K\"{a}hler metrics  on the Cartan-Hartogs domain
$\Omega(\mu,d_0)$.

\begin{Lemma}\label{Le:7.1}
Let $\mathcal{G}$ be the set of mappings generated by \eqref{e7.3}, $\Phi$ be a  K\"{a}hler potential on the Cartan-Hartogs domain $\Omega(\mu,d_0)$. If for all $\Upsilon\in \mathcal{G}$,
\begin{equation*}
  \partial\bar{\partial}(\Phi\circ\Upsilon)=\partial\bar{\partial}\Phi,
\end{equation*}
then there are a unique real number $\nu>0$ and a unique real function $F$ with $F(0)=0$ such that
\begin{equation*}
 \partial\bar{\partial}(\nu\phi+F(\rho))=\partial\bar{\partial}\Phi,
\end{equation*}
where
\begin{equation*}
  \phi(z)=-\mu\ln N(z,\overline{z}),\; \rho=e^{\phi(z)}\|w\|^2.
\end{equation*}
\end{Lemma}

\begin{proof}[Proof]

$\mathbf{Step 1.}$ We prove that there exist a number $\nu\in\mathbb{C}$ and a function $F$ satisfying
\begin{equation*}
 \partial\bar{\partial}(\nu\phi+F(\rho))=\partial\bar{\partial}\Phi.
\end{equation*}

Let $Z=(z,w)$ and
\begin{equation*}
\mathbf{w}=\frac{w}{N(z,\bar{z})^{\frac{\mu}{2}}}=e^{\frac{1}{2}\phi(z)}w,\;\mathbf{w_0}=\frac{w_0}{N(z_0,\bar{z_0})^{\frac{\mu}{2}}}=e^{\frac{1}{2}\phi(z_0)}w_0.
\end{equation*}

By
\begin{equation*}
 \partial\bar{\partial}(\Phi\circ\Upsilon)=\partial\bar{\partial}\Phi
\end{equation*}
and
\begin{equation*}
  \frac{\partial^2(\Phi\circ\Upsilon)}{\partial Z^t\partial\overline{Z}}(Z)=\frac{\partial\Upsilon}{\partial Z^t}(Z)\frac{\partial^2\Phi}{\partial Z^t\partial\overline{Z}}(\Upsilon(Z))\overline{\left(\frac{\partial\Upsilon}{\partial Z^t}\right)^t}(Z),
\end{equation*}
we get
\begin{equation*}
  \frac{\partial^2\Phi}{\partial Z^t\partial\overline{Z}}(Z)=\frac{\partial\Upsilon}{\partial Z^t}(Z)\frac{\partial^2\Phi}{\partial Z^t\partial\overline{Z}}(\Upsilon(Z))\overline{\left(\frac{\partial\Upsilon}{\partial Z^t}\right)^t}(Z).
\end{equation*}

Let
\begin{equation*}
\left(
  \begin{array}{cc}
    A_{11} & A_{12} \\
    A_{21} & A_{22} \\
  \end{array}
\right)
 :=\left(
                                                                     \begin{array}{cc}
                                                                      \frac{\partial^2\Phi}{\partial z^t\partial \bar z}  & \frac{\partial^2\Phi}{\partial z^t\partial \bar w} \\
                                                                       \frac{\partial^2\Phi}{\partial w^t\partial \bar z} & \frac{\partial^2\Phi}{\partial w^t\partial \bar w} \\
                                                                     \end{array}
                                                                   \right),
\end{equation*}
using
\begin{equation*}
 \frac{\partial\Upsilon}{\partial Z^t}(Z)=\left(
                                                  \begin{array}{ll}
                                                    \frac{\partial\gamma}{\partial z^t}(z) & \frac{\partial\psi}{\partial z^t}wU\\
                                                    0 & \psi(z)U \\
                                                  \end{array}
                                                \right),
\end{equation*}
we obtain
\begin{eqnarray}
\nonumber  A_{11}(Z) &=& \frac{\partial\gamma}{\partial z^t}A_{11}(\Upsilon(Z))\overline{\left(\frac{\partial\gamma}{\partial z^t}\right)^t}+\frac{\partial\psi}{\partial z^t}wUA_{21}(\Upsilon(Z))\overline{\left(\frac{\partial\gamma}{\partial z^t}\right)^{t}} \\
\label{e7.5}   & & +\frac{\partial\gamma}{\partial z^t}A_{12}(\Upsilon(Z))\overline{\left(\frac{\partial\psi}{\partial z^t}wU\right)^t}+\frac{\partial\psi}{\partial z^t}wUA_{22}(\Upsilon(Z))\overline{\left(\frac{\partial\psi}{\partial z^t}wU\right)^t},
\end{eqnarray}
\begin{equation}\label{e7.6}
  A_{12}(Z)=\overline{\psi(z)}\frac{\partial\gamma}{\partial z^t}A_{12}(\Upsilon(Z))\overline{U^t}+\overline{\psi(z)}\frac{\partial\psi}{\partial z^t}wUA_{22}(\Upsilon(Z))\overline{U^t},
\end{equation}
\begin{equation}\label{e7.7}
  A_{21}(Z)=\psi(z)UA_{21}(\Upsilon(Z))\overline{\left(\frac{\partial\gamma}{\partial z^t}\right)^t}+\psi(z)UA_{22}(\Upsilon(Z))\overline{\left(\frac{\partial\psi}{\partial z^t}wU\right)^t}
\end{equation}
and
\begin{equation}\label{e7.8}
   A_{22}(Z)=|\psi(z)|^2UA_{22}(\Upsilon(Z))\overline{U^t}.
\end{equation}

For $Z=(z_0,w_0)$, since
$$\gamma(z_0)=0,\Upsilon(z_0,w_0)=(0,\mathbf{w_0}U),\;\psi(z_0)=e^{\frac{1}{2}\phi(z_0)},\;\frac{\partial\psi}{\partial z^t}(z_0)=e^{\frac{1}{2}\phi(z_0)}\frac{\partial\phi}{\partial z^t}(z_0),$$
we have
\begin{eqnarray}
\nonumber  A_{11}(z_0,w_0) &=& \frac{\partial\gamma}{\partial z^t}(z_0)A_{11}(0,\mathbf{w_0}U)\overline{\left(\frac{\partial\gamma}{\partial z^t}(z_0)\right)^t}+\frac{\partial\psi}{\partial z^t}(z_0)w_0UA_{21}(0,\mathbf{w_0}U)\overline{\left(\frac{\partial\gamma}{\partial z^t}(z_0)\right)^{t}} \\
\nonumber   & & +\frac{\partial\gamma}{\partial z^t}(z_0)A_{12}(0,\mathbf{w_0}U)\overline{\left(\frac{\partial\psi}{\partial z^t}(z_0)wU\right)^t}\\
\label{e7.9}  & &+e^{\phi(z_0)}\frac{\partial\phi}{\partial z^t}(z_0)w_0UA_{22}(0,\mathbf{w_0}U)\overline{U^t}\;\overline{w_0^t}\frac{\partial\phi}{\partial\bar{z}}(z_0),
\end{eqnarray}
\begin{equation}\label{e7.10}
  A_{12}(z_0,w_0)=e^{\frac{1}{2}\phi(z_0)}\frac{\partial\gamma}{\partial z^t}(z_0)A_{12}(0,\mathbf{w_0}U)\overline{U^t}+e^{\phi(z_0)}\frac{\partial\phi}{\partial z^t}(z_0)w_0UA_{22}(0,\mathbf{w_0}U)\overline{U^t},
\end{equation}
\begin{equation}\label{e7.11}
  A_{21}(z_0,w_0)=e^{\frac{1}{2}\phi(z_0)} UA_{21}(0,\mathbf{w_0}U)\overline{\left(\frac{\partial\gamma}{\partial z^t}(z_0)\right)^t}+e^{\phi(z_0)}UA_{22}(0,\mathbf{w_0}U)\overline{U^t}\;\overline{w_0^t}\frac{\partial\phi}{\partial\bar{z}}(z_0)
\end{equation}
and
\begin{equation}\label{e7.12}
   A_{22}(z_0,w_0)=e^{\phi(z_0)}UA_{22}(0,\mathbf{w_0}U)\overline{U^t}.
\end{equation}

Let $z_0=0$ in \eqref{e7.12}, it follows that
\begin{equation*}
 A_{22}(0,w_0)=UA_{22}(0,w_0U)\overline{U^t}.
\end{equation*}
According to Lemma \ref{Le:7.2} below, there is a function $F$ such that
\begin{equation*}
 A_{22}(0,w)=F'(\|w\|^2)I_{d_0}+F''(\|w\|^2)\overline{w^t}w.
\end{equation*}
Thus
\begin{equation}\label{e7.13}
  A_{22}(z,w)=e^{\phi(z)}\left(F'(\rho)I_{d_0}+e^{\phi(z)}F''(\rho)\overline{w^t}w\right).
\end{equation}

Set $z_0=0$ and $U=I_{d_0}$ in \eqref{e7.10}, by $\phi(0)=0$ and $\frac{\partial\phi}{\partial z^t}(0)=0$, we have
\begin{equation*}
  A_{12}(0,w_0)=\frac{\partial\gamma}{\partial z^t}(0)A_{12}(0,w_0)
\end{equation*}
for all $\gamma\in \mathrm{Aut}(\Omega)$ with $\gamma(0)=0$. So $A_{12}(0,w_0)=0$. Combining \eqref{e7.10} and \eqref{e7.13}, we get
\begin{equation}\label{e7.14}
  A_{12}(z,w)=e^{\phi(z)}\left(F'(\rho)+\rho F''(\rho)\right)\frac{\partial\phi}{\partial z^t}(z)w.
\end{equation}
For the same argument, we also have
\begin{equation}\label{e7.15}
  A_{21}(z,w)=e^{\phi}(F'(\rho)+\rho F''(\rho))\overline{w^t}\frac{\partial\phi}{\partial \bar{z}}.
\end{equation}

Substituting $A_{12}(0,w)=0$, $A_{21}(0,w)=0$ and \eqref{e7.13} into \eqref{e7.9}, then
\begin{eqnarray}
\nonumber  A_{11}(z_0,w_0) &=& \frac{\partial\gamma}{\partial z^t}(z_0)A_{11}(0,\mathbf{w_0}U)\overline{\left(\frac{\partial\gamma}{\partial z^t}(z_0)\right)^t} \\
\label{e7.16}   & & +\rho(F'(\rho)+\rho F''(\rho))\frac{\partial\phi}{\partial z^t}(z_0)\frac{\partial\phi}{\partial\bar{z}}(z_0),
\end{eqnarray}
here $\gamma(z_0)=0$ and $\rho=e^{\phi(z_0)}\|w_0\|^2$.

Let $z_0=0$, \eqref{e7.16} implies that
\begin{equation*}
 A_{11}(0,w)=\frac{\partial\gamma}{\partial z^t}A_{11}(0,wU)\overline{\left(\frac{\partial\gamma}{\partial z^t}\right)^t}
\end{equation*}
for all $U\in\mathcal{U}(d_0)$, $\gamma\in \mathrm{Aut}(\Omega)$ with $\gamma(0)=0$. Therefore, by Schur's lemma, there exists a function $P$  such that
\begin{equation*}
   A_{11}(0,w)=\mu P(\|w\|^2)I_{d}.
\end{equation*}

Since
\begin{equation*}
 \partial\bar{\partial}(\phi\circ\gamma)=\partial\bar{\partial}\phi
\end{equation*}
for all $\gamma\in \mathrm{Aut}(\Omega)$, using $\frac{\partial^2\phi}{\partial z^t\partial\bar{z}}(0)=\mu I_d$, it follows that
\begin{equation*}
 \frac{\partial^2\phi}{\partial z^t\partial\bar{z}}(z_0)=\mu\frac{\partial\gamma}{\partial z^t}(z_0)\overline{\left(\frac{\partial\gamma}{\partial z^t}(z_0)\right)^t}
\end{equation*}
for all $\gamma\in \mathrm{Aut}(\Omega)$ with $\gamma(z_0)=0$. Hence, by \eqref{e7.16}, we obtain
\begin{equation}\label{e7.17}
A_{11}(z,w)= P(\rho)\frac{\partial^2\phi}{\partial z^t\partial\bar{z}}+\rho(F'(\rho)+\rho F''(\rho))\frac{\partial\phi}{\partial z^t}(z_0)\frac{\partial\phi}{\partial\bar{z}}(z_0).
\end{equation}

Now, we set $\Psi=\Phi-F(\rho)$, then
\begin{equation*}
  A_{11}=\frac{\partial^2\Psi}{\partial z^t\partial\bar{z}}+\rho F'(\rho)\frac{\partial^2\phi}{\partial z^t\partial\bar{z}}
  +\rho(F'(\rho)+\rho F''(\rho))\frac{\partial\phi}{\partial z^t}\frac{\partial\phi}{\partial\bar{z}},
\end{equation*}
\begin{equation*}
  A_{12}=\frac{\partial^2\Psi}{\partial z^t\partial\bar{w}}+e^{\phi}\left(F'(\rho)+\rho F''(\rho)\right)\frac{\partial\phi}{\partial z^t}w,
\end{equation*}
\begin{equation*}
  A_{21}=\frac{\partial^2\Psi}{\partial w^t\partial\bar{z}}+e^{\phi}(F'(\rho)+\rho F''(\rho))\overline{w^t}\frac{\partial\phi}{\partial \bar{z}}
\end{equation*}
and
\begin{equation*}
  A_{22}=\frac{\partial^2\Psi}{\partial w^t\partial\bar{w}}+e^{\phi}\left(F'(\rho)I_{d_0}+e^{\phi(z)}F''(\rho)\overline{w^t}w\right).
\end{equation*}
These combine with \eqref{e7.17}, \eqref{e7.14}, \eqref{e7.15} and \eqref{e7.13}, we get
\begin{equation}\label{e7.18}
 \frac{\partial^2\Psi}{\partial z^t\partial\bar{z}}=(P(\rho)-\rho F'(\rho))\frac{\partial^2\phi}{\partial z^t\partial\bar{z}}
\end{equation}
and
\begin{equation}\label{e7.19}
  \frac{\partial^2\Psi}{\partial z^t\partial\bar{w}}=0,\;\frac{\partial^2\Psi}{\partial w^t\partial\bar{z}}=0,\;\frac{\partial^2\Psi}{\partial w^t\partial\bar{w}}=0.
\end{equation}

From \eqref{e7.19}, the function $\Psi$ must be the following form
\begin{equation*}
 \Psi=\Psi_1(z,\bar{z})+\Psi_2(z,w)+\Psi_3(\bar{z},\bar{w}),
\end{equation*}
by \eqref{e7.18}, we have
\begin{equation*}
 \frac{\partial^2\Psi_1}{\partial z^t\partial\bar{z}}=(P(\rho)-\rho F'(\rho))\frac{\partial^2\phi}{\partial z^t\partial\bar{z}},
\end{equation*}
thinks to $\rho(z,w)=e^{\phi(z)}\|w\|^2$, so $P(\rho)-\rho F'(\rho)$ is a constant.  Finally, let $\nu=P(0)$,  then
 \begin{equation*}
 \partial\bar{\partial}(\nu\phi+F(\rho))=\partial\bar{\partial}\Phi.
\end{equation*}

$\mathbf{Step 2.}$  We prove that the function $F$ with $F(0)=0$ and the number $\nu$ are unique.

We only show that if $$ \partial\bar{\partial}(\lambda\phi+Q(\rho))\equiv 0, \;Q(0)=0,$$ then $\lambda=0$ and $Q\equiv 0$.

Since
\begin{equation*}
\frac{\partial^2(\lambda\phi+Q(\rho))}{\partial z^t\partial \bar{z}}=(\lambda+\rho Q'(\rho))\frac{\partial^2\phi}{\partial z^t\partial \bar{z}}+\rho(Q'(\rho)+\rho Q''(\rho))\frac{\partial\phi}{\partial z^t}\frac{\partial\phi}{\partial \bar{z}}=0,
\end{equation*}
let $\rho=0$, which gives
\begin{equation*}
  \lambda\frac{\partial^2\phi}{\partial z^t\partial \bar{z}}=0,
\end{equation*}
so $\lambda=0$.

By $\partial\bar{\partial}(Q(\rho))\equiv 0$, we obtain
\begin{equation*}
 \frac{\partial^2Q}{\partial w^t\partial \bar{w}}= e^{\phi}\left(Q'(\rho)I_{d_0}+e^{\phi}Q''(\rho)\bar{w}^tw\right)=0,
\end{equation*}
thus $Q'(\rho)=0$ for $d_0>1$ and $Q'(\rho)+\rho Q''(\rho)=0$ for $d_0=1$. So $Q\equiv 0$ for $d_0>1$ and $\rho Q'(\rho)=c$ for $d_0=1$.
For $d_0=1$, let $\rho=0$, we have $c=0$, hence $Q\equiv 0$.

$\mathbf{Step 3.}$  We prove that $F$ with $F(0)=0$ is a real function and $\nu>0$.

Using
\begin{equation*}
\frac{\partial^2(\nu\phi+F(\rho))}{\partial z^t\partial \bar{z}}=(\nu+\rho F'(\rho))\frac{\partial^2\phi}{\partial z^t\partial \bar{z}}+\rho(F'(\rho)+\rho F''(\rho))\frac{\partial\phi}{\partial z^t}\frac{\partial\phi}{\partial \bar{z}}
\end{equation*}
is positive definite, we get $\nu\frac{\partial^2\phi}{\partial z^t\partial \bar{z}}$ is positive definite, then $\nu>0$.

Since
\begin{equation*}
\frac{\partial^2(\nu\phi+F(\rho))}{\partial w^t\partial \bar{w}}= e^{\phi}\left(F'(\rho)I_{d_0}+e^{\phi}F''(\rho)\bar{w}^tw\right)
\end{equation*}
is positive definite, then
\begin{equation*}
  (F-\overline{F})'I_{d_0}+e^{\phi}(F-\overline{F})''\bar{w}^tw=0,
\end{equation*}
which implies $F=\overline{F}$, namely $F$ is a real function.
\end{proof}

\begin{Lemma}\label{Le:7.2}
Let $\Psi$ be the continuous differentiable function of order 2 on $\mathbb{B}^{m}$. If for all $U\in \mathcal{U}(m)$,
\begin{equation*}
  \partial\bar{\partial}(\Psi\circ U)=\partial\bar{\partial}\Psi,
\end{equation*}
namely
\begin{equation*}
  \frac{\partial^2\Psi}{\partial w^t\partial\bar{w}}(w,\bar{w})=U\frac{\partial^2\Psi}{\partial w^t\partial\bar{w}}(wU,\overline{wU})\overline{U^t},
\end{equation*}
then there is a function $F$ such that
\begin{equation*}
 \partial\bar{\partial}(F(\|w\|^2))=\partial\bar{\partial}\Psi.
\end{equation*}
\end{Lemma}

\begin{proof}[Proof]
We only prove for the case $m=1,2$, the proof of the case $m>2$ is the same as the proof of the case $m=2$, we omit its the proof.

(i) For $m=1$, we have
\begin{equation*}
  \frac{\partial^2\Psi}{\partial w^t\partial\bar{w}}(w,\bar{w})=\frac{\partial^2\Psi}{\partial w^t\partial\bar{w}}(e^{\sqrt{-1}\theta}w,e^{-\sqrt{-1}\theta}\bar{w}),\;\forall\;\theta\in\mathbb{R},
\end{equation*}
then there exists a function $f$ that satisfies
\begin{equation*}
  \frac{\partial^2\Psi}{\partial w^t\partial\bar{w}}(w,\bar{w})=f(|w|^2).
\end{equation*}
Let
\begin{equation*}
  Q(x)=\left\{\begin{array}{ll}
                \frac{1}{x}\int_0^xf(t)dt, & x\in(0,1), \\
                f(0), & x=0
              \end{array}
  \right.
\end{equation*}
and
\begin{equation*}
  F(x)=\int_0^xQ(t)dt,\;x\in[0,1).
\end{equation*}
It easy to see that $F$ satisfies
\begin{equation*}
 \partial\bar{\partial}(F(\|w\|^2))=\partial\bar{\partial}\Psi.
\end{equation*}

(ii) For $m=2$, let
\begin{equation*}
  f(w,\bar{w})=\left(
                      \begin{array}{cc}
                        f_{11} & f_{12} \\
                        f_{21} & f_{22} \\
                      \end{array}
                    \right)(w,\bar{w})
  = \frac{\partial^2\Psi}{\partial w^t\partial\bar{w}}(w),
\end{equation*}
thus
\begin{equation*}
  f(wU,\bar{w}\overline{U})=\overline{U^t}f(w,\bar{w})U,\;\forall \;U\in \mathcal{U}(2).
\end{equation*}
The following we will prove that there exist  functions $r$ and $s$ satisfying
\begin{equation}\label{e7.20.0}
   f(w,\bar{w})=r(\|w\|^2)I_2+s(\|w\|^2)\overline{w^t}w.
\end{equation}

Let $U=e^{tX}\in \mathcal{U}(2)$, $X=(X_{ij})_{1\leq i,j\leq 2}$ and $t\in\mathbb{R}$. Differentiating the left and right-hand sides at $t=0$ below,
\begin{equation*}
   f(we^{tX},\bar{w}e^{t\overline{X}})=e^{-tX}f(w,\bar{w})e^{tX},
\end{equation*}
we find
\begin{equation}\label{e7.20}
  \sum_{ij=1}^2\left(w_iX_{ij}\frac{\partial f}{\partial w_j}-\overline{w_j}X_{ij}\frac{\partial f}{\partial \overline{w_i}}\right)=fX-Xf.
\end{equation}

Setting
$$X=\sqrt{-1}\left(
                      \begin{array}{cc}
                        1 & 0 \\
                        0 & 0 \\
                      \end{array}
                    \right),\;\sqrt{-1}\left(
                      \begin{array}{cc}
                        0 & 0 \\
                        0 & 1 \\
                      \end{array}
                    \right),\;\sqrt{-1}\left(
                      \begin{array}{cc}
                        0 & 1 \\
                        1 & 0 \\
                      \end{array}
                    \right),\;\left(
                      \begin{array}{cc}
                        0 & -1 \\
                        1 & 0 \\
                      \end{array}
                    \right)
$$
in \eqref{e7.20}, we have
\begin{equation}\label{e7.21}
  w_1\frac{\partial f}{\partial w_1}-\overline{w_1}\frac{\partial f}{\partial\overline{w_1}}=\left(
                                                                                               \begin{array}{cc}
                                                                                                 0 & -f_{12} \\
                                                                                                 f_{21} & 0 \\
                                                                                               \end{array}
                                                                                             \right),
\end{equation}
\begin{equation}\label{e7.22}
  w_2\frac{\partial f}{\partial w_2}-\overline{w_2}\frac{\partial f}{\partial\overline{w_2}}=\left(
                                                                                               \begin{array}{cc}
                                                                                                 0 & f_{12} \\
                                                                                                 -f_{21} & 0 \\
                                                                                               \end{array}
                                                                                             \right),
\end{equation}
\begin{equation}\label{e7.23}
  w_1\frac{\partial f}{\partial w_2}+ w_2\frac{\partial f}{\partial w_1}-\overline{w_1}\frac{\partial f}{\partial\overline{w_2}}-\overline{w_2}\frac{\partial f}{\partial\overline{w_1}}=\left(
                                                                                               \begin{array}{cc}
                                                                                                 f_{12}-f_{21} & f_{11}-f_{22} \\
                                                                                                 f_{22}-f_{11} & f_{21}-f_{12} \\
                                                                                               \end{array}
                                                                                             \right),
\end{equation}
\begin{equation}\label{e7.24}
  -w_1\frac{\partial f}{\partial w_2}+ w_2\frac{\partial f}{\partial w_1}-\overline{w_1}\frac{\partial f}{\partial\overline{w_2}}+\overline{w_2}\frac{\partial f}{\partial\overline{w_1}}=\left(
                                                                                               \begin{array}{cc}
                                                                                                 f_{12}+f_{21} & -f_{11}+f_{22} \\
                                                                                                 -f_{11}+f_{22} & -f_{12}-f_{21} \\
                                                                                               \end{array}
                                                                                             \right),
\end{equation}
 respectively.

 Using \eqref{e7.23} and \eqref{e7.24}, we get
\begin{equation}\label{e7.25}
 w_2\frac{\partial f}{\partial w_1}-\overline{w_1}\frac{\partial f}{\partial\overline{w_2}}=\left(
                                                                                               \begin{array}{ll}
                                                                                                 f_{12} & 0 \\
                                                                                                 f_{22}-f_{11} & -f_{12} \\
                                                                                               \end{array}
                                                                                             \right)
\end{equation}
and
\begin{equation}\label{e7.26}
 w_1\frac{\partial f}{\partial w_2}-\overline{w_2}\frac{\partial f}{\partial\overline{w_1}}=\left(
                                                                                               \begin{array}{ll}
                                                                                                 -f_{21} & f_{11}-f_{22} \\
                                                                                                 0 &  f_{21} \\
                                                                                               \end{array}
                                                                                             \right).
\end{equation}

It is well known that the solutions of the equations
\begin{equation*}
  x\frac{\partial g_1 }{\partial x}- y\frac{\partial g_1 }{\partial y}=0,\; x\frac{\partial g_2 }{\partial x}- y\frac{\partial g_2 }{\partial y}=g_2
\end{equation*}
are
\begin{equation*}
  g_1=h_1(xy),\;g_2=xh_2(xy),
\end{equation*}
 respectively. Then from \eqref{e7.21} and \eqref{e7.22}, $f$ may be written as
\begin{equation}\label{e7.27}
  f(w,\bar{w})=\left(
                 \begin{array}{ll}
                   h_{11}(|w_1|^2,|w_2|^2) & \overline{w_1}w_2h_{12}(|w_1|^2,|w_2|^2) \\
                   \overline{w_2}w_1h_{21}(|w_1|^2,|w_2|^2) & h_{22}(|w_1|^2,|w_2|^2) \\
                 \end{array}
               \right).
\end{equation}

Let $u=w_1\overline{w_1}$, $v=w_2\overline{w_2}$. Substituting \eqref{e7.27} into \eqref{e7.25} and \eqref{e7.26}, we get
\begin{eqnarray*}
   & &  \left(
     \begin{array}{ll}
     \overline{w_1}w_2\left(\frac{\partial h_{11}}{\partial u}-\frac{\partial h_{11}}{\partial v}\right)   & (\overline{w_1}w_2)^2\left(\frac{\partial h_{12}}{\partial u}-\frac{\partial h_{12}}{\partial v}\right) \\
      (|w_2|^2-|w_1|^2)h_{21}+|w_1w_2|^2\left(\frac{\partial h_{21}}{\partial u}-\frac{\partial h_{21}}{\partial v}\right)  & \overline{w_1}w_2\left(\frac{\partial h_{22}}{\partial u}-\frac{\partial h_{22}}{\partial v}\right) \\
     \end{array}
   \right) \\
   &=&  \left(
          \begin{array}{ll}
           \overline{w_1}w_2h_{12}  & 0 \\
             h_{22}-h_{11}& -\overline{w_1}w_2h_{12}  \\
          \end{array}
        \right)
\end{eqnarray*}
and
\begin{eqnarray*}
   & &  \left(
     \begin{array}{ll}
     \overline{w_2}w_1\left(\frac{\partial h_{11}}{\partial v}-\frac{\partial h_{11}}{\partial u}\right)   & (|w_1|^2-|w_2|^2)h_{12}+|w_1w_2|^2\left(\frac{\partial h_{12}}{\partial v}-\frac{\partial h_{12}}{\partial u}\right) \\
      (\overline{w_2}w_1)^2\left(\frac{\partial h_{21}}{\partial v}-\frac{\partial h_{21}}{\partial u}\right)  & \overline{w_2}w_1\left(\frac{\partial h_{22}}{\partial v}-\frac{\partial h_{22}}{\partial u}\right) \\
     \end{array}
   \right) \\
   &=&  \left(
          \begin{array}{ll}
           -\overline{w_2}w_1h_{21}  &  h_{11}-h_{22} \\
           0 & \overline{w_2}w_1h_{21}  \\
          \end{array}
        \right).
\end{eqnarray*}
That is
\begin{eqnarray}
\label{e7.28}         \frac{\partial h_{11}}{\partial u}-\frac{\partial h_{11}}{\partial v} &=&h_{12},  \\
\label{e7.29}        \frac{\partial h_{12}}{\partial u}-\frac{\partial h_{12}}{\partial v}  &=&0,  \\
\label{e7.30}          (v-u)h_{21}+uv\left(\frac{\partial h_{21}}{\partial u}-\frac{\partial h_{21}}{\partial v}\right) &=& h_{22}-h_{11}, \\
\label{e7.31}         \frac{\partial h_{22}}{\partial u}-\frac{\partial h_{22}}{\partial v} &=&  =-h_{12},
\end{eqnarray}
and
\begin{eqnarray}
\label{e7.32}  \frac{\partial h_{11}}{\partial u}-\frac{\partial h_{11}}{\partial v}  &=&h_{21},  \\
\label{e7.33}    \frac{\partial h_{21}}{\partial u}-\frac{\partial h_{21}}{\partial v} &=&0,  \\
\label{e7.34}  (v-u)h_{12}+uv\left(\frac{\partial h_{12}}{\partial u}-\frac{\partial h_{12}}{\partial v}\right) &=& h_{22}-h_{11}, \\
\label{e7.35}  \frac{\partial h_{22}}{\partial u}-\frac{\partial h_{22}}{\partial v}  &=& -h_{21}.
\end{eqnarray}

 Combining \eqref{e7.28}, \eqref{e7.29}, \eqref{e7.32} and \eqref{e7.33}, we obtain that there exists a function $s$ such that
 \begin{equation}\label{e7.36}
   h_{12}(u,v)=h_{21}(u,v)=s(u+v).
 \end{equation}

Comparing \eqref{e7.31}, \eqref{e7.32}, we have
\begin{equation*}
    \frac{\partial (h_{11}+h_{22})}{\partial u}-\frac{\partial (h_{11}+h_{22})}{\partial v}=0,
\end{equation*}
so there is a function $p$ such that
\begin{equation}\label{e7.37}
  h_{11}+h_{22}=p(u+v).
\end{equation}

Substituting \eqref{e7.36} into \eqref{e7.30},  it gives that
\begin{equation}\label{e7.38}
  h_{11}-h_{22}=(u-v)s(u+v).
\end{equation}

By \eqref{e7.37} and \eqref{e7.38}, we have
\begin{equation}\label{e7.39}
  h_{11}=r(u+v)+us(u+v),\; h_{22}=r(u+v)+vs(u+v),
\end{equation}
here
\begin{equation*}
  r(x)=\frac{1}{2}p(x)-\frac{1}{2}xs(x).
\end{equation*}

Substituting \eqref{e7.36} and  \eqref{e7.39} into \eqref{e7.27}, we obtain \eqref{e7.20.0}.

Now let
$$q(x)=\int_0^x(x-t)s(t)dt\;(0\leq x<1) ,\;\Psi_1(w,\bar{w})=\Psi(w,\bar{w})-q(\|w\|^2).$$
Thus \eqref{e7.20.0} gives that
\begin{equation}\label{e7.41}
  \frac{\partial^2\Psi_1}{\partial w_1\partial \overline{w_1}}(w,\bar{w})=\frac{\partial^2\Psi_1}{\partial w_2\partial \overline{w_2}}(w,\bar{w})=r(\|w\|^2)-q'(\|w\|^2)
\end{equation}
and
\begin{equation}\label{e7.42}
  \frac{\partial^2\Psi_1}{\partial w_1\partial \overline{w_2}}=\frac{\partial^2\Psi_1}{\partial w_2\partial \overline{w_1}}=0.
\end{equation}

From \eqref{e7.42} and \eqref{e7.41}, it follows that
\begin{equation*}
  \Psi_1(w,\bar{w})=H_1(w_1,\overline{w_1})+H_2(w_2,\overline{w_2})+H_3(w_1,w_2)+H_4({\overline{w_1},\overline{w_2}})
\end{equation*}
and
\begin{equation*}
  \frac{\partial^2 H_1}{\partial w_1\partial \overline{w_1}}(w_1,\overline{w_1})=r(\|w\|^2)-q'(\|w\|^2),
\end{equation*}
This indicates that $r(\|w\|^2)-q'(\|w\|^2)$ is a constant.

Let $c=r(\|w\|^2)-q'(\|w\|^2)$ and $F(x)=cx+q(x)$, we have
\begin{equation*}
 \partial\bar{\partial}(F(\|w\|^2))=\partial\bar{\partial}\Psi.
\end{equation*}
\end{proof}

Now we give a proof of Theorem \ref{Th:1.1}.

\begin{proof}[Proof of Theorem \ref{Th:1.1}]
From Lemma \ref{Le:7.1}, the K\"{a}hler potential $\Phi$ of $\mathcal{G}$-invariant the K\"{a}hler  metric $g$ can be  selected as
\begin{equation*}
  \Phi=\nu\phi+\nu F(\rho),\;\rho=e^{\phi(z)}\|w\|^2.
\end{equation*}

Let $g_F$ be a  K\"{a}hler  metric  on the domain $\Omega(\mu,d_0)$ associated with  the K\"{a}hler form
$$\omega_{g_F}=\frac{\sqrt{-1}}{2\pi}\partial\overline{\partial}(\phi+F(\rho)),$$
 namely $g=\nu g_F$, thus $\mathbf{a}_1(g)=\frac{1}{\nu}\mathbf{a}_1(g_F)$ and $\mathbf{a}_2(g)=\frac{1}{\nu^2}\mathbf{a}_2(g_F)$, here $\mathbf{a}_j(g)$ are the coefficients of the Bergman function expansion with respect to the metric $g$.

It is easy to see that completeness of the metric $g$ is the same as with the metric $g_F$

 We denote the rank $r$, the characteristic multiplicities $a,b$, the dimension $d$,  the genus $p$, and the generic norm $N(z,\overline{w})$ for the Cartan domain $\Omega$. Then the Bergman functions for the Cartan domain $\Omega$  with respect to the potential $\phi(z)=-\mu\ln N(z,\bar{z})$ are
\begin{equation}\label{e7.43}
  \epsilon(\alpha;z)=\frac{1}{\mu^d}\prod_{j=1}^r\left(\mu\alpha-p+1+(j-1)\frac{a}{2}\right
   )_{1+b+(r-j)a}
\end{equation}
for all $\alpha>\frac{p-1}{\mu}$, where $(x)_k=\frac{\Gamma(x+k)}{\Gamma(x)}$. In particular
\begin{equation*}
  \epsilon(\alpha;z)=\prod_{j=1}^d\left(\alpha-\frac{j}{\mu}\right)
\end{equation*}
for $\Omega=\mathbb{B}^d$ and $\alpha>\frac{d}{\mu}$.

Using \eqref{e7.43} and Lemma 3.3 of \cite{FT}, we obtain that the first two coefficients of the Bergman function expansion for the Cartan domain $\Omega$ with respect to the potential $\phi(z)=-\mu\ln N(z,\bar{z})$ are
\begin{equation}\label{eq2.50}
   a_1=-\frac{dp}{2\mu}
\end{equation}
and
\begin{eqnarray}
\nonumber a_2  &=& \frac{1}{2\mu^2}\left\{\frac{d^2p^2}{4}-\frac{r(p-1)p(2p-1)}{6}+\frac{r(r-1)a(3p^2-3p+1)}{12}\right. \\
\label{eq2.51}&&-\left.\frac{(r-1)r(2r-1)a^2(p-1)}{24}+\frac{r^2(r-1)^2a^3}{48}\right\}.
\end{eqnarray}

Notice that $\Omega=\mathbb{B}$, $p=2$, $r=1$ and $a=2$ for $d=1$, so by \eqref{eq2.50} and \eqref{eq2.51} we get
$$a_1=-\frac{1}{\mu},\;a_2=0.$$

From the proof of Theorem 1.3 of \cite{FT}, we have that
$$a_1=-\frac{d(d+1)}{2},\;a_2=\frac{(d-1)d(d+1)(3d+2)}{24}$$
if and only if $\Omega=\mathbb{B}^d$ and $\mu=1$.

According to Theorem \ref{apth:3.3}, both the coefficients $\mathbf{a}_1(g_F)$ and $\mathbf{a}_2(g_F)$ are constants iff
for $d>1$,
$$\Omega=\mathbb{B}^d,\mu=1,\omega_{g_F}=-\frac{\sqrt{-1}}{2\pi}\partial\overline{\partial}\ln(1-\|z\|^2-\|w\|^2);$$
for $d=1$,
$$\Omega=\mathbb{B},\omega_{g_F}=-\frac{\sqrt{-1}}{2\pi}\partial\overline{\partial}\left\{\mu\ln(1-|z|^2)+\frac{(d_0+1)\mu}{d_0\mu+1}\ln\left(1-\frac{\|w\|^2}{(1-|z|^2)^{\mu}}\right)\right\}.$$
This completes the proof.
\end{proof}

\vskip 20pt

 \noindent\textbf{Acknowledgments}\quad The author would like to thank the referees for many helpful suggestions.  The author was
supported by the Scientific Research Fund of Sichuan Provincial Education Department (No.18ZB0272).


\addcontentsline{toc}{section}{References}
\phantomsection
\renewcommand\refname{References}
\small{
}
\clearpage

\begin{thebibliography}{99}
\setlength{\parskip}{0pt}

\bibitem{ABP}Ahn H., Byun J. and  Park J.D.: Automorphisms of the Hartogs type domains over classical symmetric domains. International Journal of Mathematics, \textbf{23}(9),  1250098 (11 pages) (2012)

\bibitem{Berezin}Berezin F. A.: Quantization, Math. USSR Izvestiya, \textbf{8}, 1109-1163 (1974)

\bibitem{B-B-S}Berman, R., Berndtsson, B., Sj\"ostrand, J.: A direct approach to Bergman kernel asymptotics for positive line bundles. Ark. Mat.,  \textbf{46}(2), 197-217 (2008)

\bibitem{B-F-T} Bi E.C., Feng Z.M., Tu Z.H.: Balanced metrics on the Fock-Bargmann-Hartogs domains. Ann. Global Anal. Geom., \textbf{49}, 349-359 (2016)


\bibitem{CGR}Cahen M., Gutt S., Rawnsley J.: Quantization of K\"{a}hler manifolds. I: Geometric interpretation of Berezin's quantization. J. Geom. Phys. \textbf{7}, 45-62 (1990)


\bibitem{Calabi}Calabi E.: Extremal K¡§ahler metrics, in Seminar on Differential Geometry, Vol. 16, Annals of Mathematics Studies, Vol. 102, Princeton University Press, Princeton,  pp. 259-290 (1982)

\bibitem{Cat}Catlin D.: The Bergman kernel and a theorem of Tian. Analysis and geometry in several complex variables (Katata, 1997), Trends Math., Birkh\"{a}user Boston, Boston, MA, pp. 1-23 (1999)


\bibitem{Dai-Liu-Ma}Dai X., Liu K., Ma X.: On the asymptotic expansion of Bergman kernel. J. Differential Geom.,  \textbf{72},  1-41 (2006)


\bibitem{Donaldson}Donaldson S.: Scalar curvature and projective embeddings, I.  J. Differential Geom. \textbf{59}, 479-522 (2001)

\bibitem{E0}Engli\v{s} M.: Berezin Quantization and Reproducing Kernels on Complex Domains. Trans. Amer. Math. Soc.  \textbf{348}, 411-479 (1996)

\bibitem{E1}Engli\v{s} M.: A Forelli-Rudin construction and asymptotics of weighted Bergman kernels. J. Funct. Anal., \textbf{177}, 257-281 (2000)

\bibitem{E2}Engli\v{s} M.: The asymptotics of a Laplace integral on a K\"{a}hler manifold. J. Reine Angew. Math., \textbf{528}, 1-39 (2000)


\bibitem{FKKLR}Faraut, J., Kaneyuki, S., Kor\'{a}nyi, A., Lu, Q.K., Roos, G.: Analysis and Geometry on Complex Homogeneous Domains.
Progress in mathematics, Vol. 185, Birkh\"{a}user, Boston (2000)

\bibitem{F1}Feng Z.M.: On the first two coefficients of the Bergman function expansion for radial metrics. Journal of Geometry and Physics, \textbf{119}, 256-271 (2017)

\bibitem{FT}Feng Z.M., Tu Z.H.: On canonical metrics on Cartan-Hartogs domains. Math. Z., \textbf{278}, 301-320 (2014)

\bibitem{FT2}Feng Z.M., Tu Z.H.: Balanced  metrics on some Hartogs type domains over bounded symmetric domains.  Ann. Global Anal. Geom., \textbf{47}(4), 305-333 (2015)

\bibitem{HCY} Hsiao C.Y.: On the coefficients of the asymptotic expansion of the kernel of Berezin-Toeplitz quantization. Ann. Global Anal. Geom. \textbf{42}(2), 207-245 (2012)

\bibitem{HM2014} Hsiao C.Y., Marinescu G.: Asymptotics of spectral function of lower energy forms and Bergman kernel of semi-positive and big line bundles. Comm. Anal. Geom. \textbf{22}(1), 1-108 (2014)

\bibitem{Hua}Hua, L.K.: Harmonic Analysis of Functions of Several Complex Variables in the Classical Domains. Amer. Math. Soc., Providence, RI (1963)


\bibitem{Hwang-Singer}Hwang A., Singer M. : A momentum construction for circle-invariant  K\"{a}hler metrics. Transactions of the American Mathematical Society,  \textbf{354}(6), 2285-2325 (2002)

\bibitem{Loi1}Loi A.: Regular quantizations of K\"{a}hler manifolds and constant scalar curvature metrics. Journal of Geometry and Physics, \textbf{53}(3), 354-364 (2005)

\bibitem{Loi-Zed2} Loi A., Zedda M.: On the coefficients of TYZ expansion of locally Hermitian symmetric spaces, Manuscripta Mathematica, \textbf{148}, 303-315 (2015)

\bibitem{LZ}Loi A., Zedda M.: Balanced metrics on Cartan and Cartan-Hartogs domains. Math. Z. \textbf{270}, 1077-1087 (2012)


\bibitem{Loi-Zud} Loi A., Zuddas F.: Engli\v{s} expansion for Hartogs domains, International Journal of Geometric Methods in Modern Physics, \textbf{6},  233-240 (2009)

\bibitem{Lu}Lu Z.: On the lower order terms of the asymptotic expansion of Tian-Yau-Zelditch. Amer. J. Math., \textbf{122}(2), 235-273 (2000)

\bibitem{Lui} Lui\'c, S.: Balanced Metrics and Noncommutative K\"ahler Geometry. Symmetry, Integrability and Geometry: Methods and Applications \textbf{6}, 069, 15 pages (2010)

\bibitem{MA2010} Ma X.: Geometric quantization on K\"ahler and symplectic manifolds. Proceedings of the international congress of mathematicians (ICM 2010), Hyderabad, India, August, 2010. Vol II, 785-810.

\bibitem{MM07}Ma X. and  Marinescu G.: Holomorphic Morse inequalities and Bergman kernels. Progress in Mathematics, Vol. 254, Birkh\v{a}user Boston Inc., Boston, MA (2007)

\bibitem{MM08}Ma X. and  Marinescu G.: Generalized Bergman kernels on symplectic manifolds. Adv. Math., \textbf{217}(4), 1756-1815 (2008)

\bibitem{MM12}Ma X. and  Marinescu G.: Berezin-Toeplitz quantization on K\"ahler manifolds. J. Reine Angew. Math. \textbf{662}, 1-56 (2012)

\bibitem{RWYZ} Wang A., Yin W.P., Zhang L.Y., Roos G.: The K\"{a}hler-Einstein  metric for some Hartogs domains over bounded symmetric domains. Science in China Series A: Mathematics \textbf{49}(9), 1175-1210 (2006)

\bibitem{TW}Tu Z.H., Wang L.: Rigidity of proper holomorphic mappings between equidimensional Hua domains. Math. Ann., \textbf{363}, 1-34 (2015)


\bibitem{X}Xu H.: A closed formula for the asymptotic expansion of the Bergman kernel. Commun. Math. Phys., \textbf{314}, 555-585 (2012)


\bibitem{Zed}Zedda M.: Canonical metrics on Cartan-Hartogs domains. International Journal of Geometric Methods in Modern Physics, \textbf{9}(1), 1250011 (13 pages) (2012)

\bibitem{Zeld}Zelditch S.: Szeg\"{o} kernels and a theorem of Tian. Internat. Math. Res. Notices, \textbf{6}, 317-331 (1998)

\end{thebibliography}
\end{document}